\documentclass[10pt, twoside]{amsart}
\usepackage{enumerate}
\usepackage{lingmacros}
\usepackage{tree-dvips}
\usepackage{amsfonts}
\usepackage{amssymb}
\usepackage{natbib}
\usepackage{mathtools}
\newcommand*{\QEDB}{\hfill\ensuremath{\square}}%
\DeclareMathOperator*{\argmin}{arg\,min}
\usepackage[colorlinks,citecolor=blue,urlcolor=blue]{hyperref}

\usepackage{natbib}
\usepackage{setspace}
\usepackage{amsmath}
\usepackage{longtable}
\usepackage{booktabs}
\usepackage{caption}
\usepackage{subcaption}
\usepackage{lscape}
\usepackage{graphicx}
\usepackage{amsthm}
\newtheorem{theo}{Theorem}[section]

\newtheorem{ass}[theo]{Assumption}
\newtheorem{prop}[theo]{Proposition}
\newtheorem{lem}[theo]{Lemma}
\newtheorem{defi}[theo]{Definition}
\theoremstyle{definition}
\newtheorem{exam}[theo]{Example}
\newtheorem{rem}[theo]{Remark}

\begin{document}

\title[Degrees of freedom for piecewise Lipschitz estimators]{Degrees of freedom for piecewise Lipschitz estimators}
\author[F. R. Mikkelsen]{Frederik Riis Mikkelsen}
\address{Department of Mathematical Sciences,
University of Copenhagen,
Universitetsparken 5,
2100 Copenhagen \O,
Denmark}

\email[Corresponding author]{frm@math.ku.dk}
\author[N. R. Hansen]{Niels Richard Hansen} 
\email{Niels.R.Hansen@math.ku.dk}



\begin{abstract}
  A representation of the degrees of freedom akin to Stein's lemma is
  given for a class of estimators of a mean value parameter in
  $\mathbb{R}^n$. Contrary to previous results our representation
  holds for a range of discontinues estimators. It shows that even
  though the discontinuities form a Lebesgue null set, they cannot be
  ignored when computing degrees of freedom.  Estimators with
  discontinuities arise naturally in regression if data driven
  variable selection is used. Two such examples, namely best subset
  selection and lasso-OLS, are considered in detail in this paper. For
  lasso-OLS the general representation leads to an estimate of the
  degrees of freedom based on the lasso solution path, which in turn
  can be used for estimating the risk of lasso-OLS. A similar
  estimate is proposed for best subset selection.  The usefulness of
  the risk estimates for selecting the number of variables is
  demonstrated via simulations with a particular focus on lasso-OLS.
\end{abstract}

\maketitle

\noindent


\section{Introduction}

Representations of the effective dimension of a statistical model have
been studied extensively in many different frameworks. For classical
model selection criteria such as AIC and Mallows's $C_p$ the dimension
of the parameter space is used to adjust the empirical risk for its
optimism so as to provide a fair model score across different dimensions. A
number of extensions to models or methods without a well defined
dimension exist, such as the trace of the smoother matrix for scatter
plot smoothers, see e.g. \cite{Hastie:1990}, and the use of the
divergence of a sufficiently differentiable estimator based on Stein's
lemma as described in \cite{Efron:2004}. Stein's lemma was used by
\cite{zou2007} and \cite{tibshirani2012} to demonstrate that for the lasso
estimator in a linear regression model with Gaussian errors, the
number of estimated non-zero parameters is an appropriate estimate of
the effective dimension.

It is well known that neither Mallows's $C_p$ nor AIC or related
information criteria correctly adjust for the optimism that results 
from selecting one model among a number of models of equal
dimension. The usage of such methods for model selection without
adequate adjustments was called ``a quiet scandal in the statistical
community'' by \cite{Breiman:1992}, who proposed a bootstrap based
method for risk estimation as an alternative. \cite{Ye:1998} defined the notion of 
generalized degrees of freedom for an estimator of the mean in a
Gaussian model and showed how to use this number for risk
estimation. The results by Ye apply to discontinuous estimators
that involve model selection, but
his proposal for computing the degrees of freedom was similarly to Breiman's
based on refitting models to perturbed data. 

If the estimator satisfies the differentiability requirements for
Stein's lemma, Lemma 2 in \cite{stein1981}, the divergence of the
estimator w.r.t. the data is an unbiased estimate of the degrees of
freedom in the generalized sense of \cite{Ye:1998}. This was used by
\cite{donoho1995}, \cite{Meyer:2000}, \cite{zou2007}, \cite{Kato:2009} and
\cite{tibshirani2012} among others to derive formulas for the degrees
of freedom of estimators that are Lipschitz continuous. 

For estimators
with discontinuities Stein's lemma generally breaks down and the
divergence will not be an unbiased estimate of the degrees of
freedom. Note that an estimator can be 
continuous or even differentiable almost everywhere -- it can be a projection locally --
and still be defined globally in such a way that it has non-ignorable
discontinuities. This is, in particular, the case in regression when
data adaptive variable selection is used to select among a number of
projection estimators. Best subset selection is one central example,
but variable selection procedures lead in general to non-ignorable
discontinuities. A variable selection procedure effectively
divides the sample space into a finite number of disjoint regions,
with the estimator being a projection, say, on each region. The resulting
estimator consisting of a selection step and a projection step will
generally be discontinuous on the boundary between two regions. 

\cite{Tibshirani:2015} recently made headway with the computation of
the degrees of freedom for some discontinuous
estimators. Specifically, he considered a linear regression model with
an orthogonal design and showed how to compute the degrees of freedom
for hard thresholding, which for orthogonal designs is equivalent to
the Lagrangian formulation of best subset selection. He also gave an
extension of Stein's lemma to some discontinuous estimators, though it
was not shown if this extension applies to subset selection
estimators. \cite{Hansen:2014} gave a different generalization of
Stein's lemma for all estimators that are metric projections onto a
closed set. This generalization applies to subset selection and other
estimators with non-convex constraints, but did not lead to a readily
computable representation of the contribution to the degrees of
freedom that are due to the discontinuities of the metric projection.

The first main contribution of this paper is the general Theorem
\ref{theo:singular_repres},
which is a version of Stein's lemma for
estimators that are locally Lipschitz continuous on each of a finite
number of open sets, whose union makes up Lebesgue almost all of
$\mathbb{R}^n$. This is a broad class of estimators containing a number of regression estimators
that include variable selection. Compared to existing results, Theorem
\ref{theo:singular_repres} holds under verifiable conditions without 
putting restrictions on the design matrix such as orthogonality. 

As a main example the lasso-OLS estimator in a
linear regression setup is investigated in detail in Section
\ref{sec:l-OLS}. The lasso-OLS estimator consists of two steps:
variable selection using lasso followed by ordinary least squares
estimation using the selected variables. This estimator was referred
to as the LARS-OLS hybrid in \cite{Efron:2004b}, and it is a limit case of 
the relaxed lasso as considered in \cite{Meinshausen:2007}. We follow 
the terminology of \cite{Buhlmann:2011}, p. 34, and call it the
lasso-OLS estimator. 

The second main contribution of this paper is a derivation of a computable estimate of the
degrees of freedom -- and thus the risk -- for lasso-OLS, which only involves the computation
of a single lasso solution path and corresponding OLS estimators along
the path. Simulation studies reported in Section
\ref{sec:sim} demonstrated that the resulting risk estimate leads to reliable model selection 
across a range of different designs and parameter settings, and 
that the risk estimate itself has smaller mean squared error than
the computationally more demanding cross-validation estimate. 

For the Lagrangian formulation of best subset selection it is also 
demonstrated that Theorem
\ref{theo:singular_repres} holds, but the situation is more
complicated than for lasso-OLS. However, it is possible to derive
an approximation, which is exact for orthogonal designs, as shown in
Section \ref{sec:best}. 

The proof of Theorem
\ref{theo:singular_repres} and some auxiliary technical results are
in the appendix. 
	
	\section{A general representation of degrees of freedom }

        Throughout the paper we consider the multivariate Gaussian model $\mathcal{N}(\mu, \sigma^2I)$ on $\mathbb{R}^n$ with $\mu$ the unknown parameter, and we let $\hat\mu:\mathbb{R}^n\rightarrow\mathbb{R}^n$ denote an estimator of $\mu$. A typical application is to linear regression estimators of the form $X \hat{\beta}$ where $X$ denotes an $n \times p$ matrix and $\hat{\beta}$ denotes an estimator of the parameters in the linear regression model. When the estimator $\hat{\beta}$ sets some of the parameters to exactly zero we say that the estimator does variable selection. The lasso, \cite{Tibshirani:1996}, is an example of a globally Lipschitz continuous estimator that does variable selection, while best subset selection is a discontinuous estimator that does variable selection. The lasso-OLS -- as studied intensively in Section \ref{sec:l-OLS} -- is another example of a discontinuous regression estimator that does variable selection. Though discontinuous regression estimators that do variable selection constitute the main motivation for the present paper, the general results are more conveniently formulated in terms of estimators of the mean $\mu$ without reference to the regression setup.

        Letting $Y \sim \mathcal{N}(\mu, \sigma^2I)$ the risk of the estimator is defined as 
$$\mathrm{Risk}(\hat{\mu}) \coloneqq E\|\mu-\hat{\mu}(Y)\|_2^2,$$
provided that $\hat{\mu}(Y)$ has finite second moment, which will thus be assumed throughout. The risk is a quantification of the error of $\hat\mu$, and tuning parameters are often chosen by minimising an estimate of the risk. Our main interest is to estimate the risk under the Gaussian model. The following definition introduces two notions of degrees of freedom that are useful when we want to estimate the risk. In the definition, $\psi(y;\mu,\sigma^2)$ denotes the density for the $\mathcal{N}(\mu, \sigma^2I)$ distribution and $\langle \cdot, \cdot \rangle$ denotes the standard inner product on $\mathbb{R}^n$. The divergence operator is also needed. It is the differential operator defined as
$$\mathrm{div}(f) = \sum_{i=1}^n \partial_i f_i$$
for $f : \mathbb{R}^n \to \mathbb{R}^n$ Lebesgue almost everywhere differentiable and with $\partial_i$ denoting the partial derivative w.r.t. the $i$th coordinate.

	\begin{defi}
		\label{defi:df} For a measurable map $\hat\mu:\mathbb{R}^n\rightarrow\mathbb{R}^n$ such that $\hat{\mu}(Y)$ has finite second moment the degrees of freedom of $\hat{\mu}$ is
		\begin{equation}
		\label{eq:df_def}
		\mathrm{df}(\hat{\mu}) \coloneqq \sum_{i=1}^{n}{\frac{\mathrm{cov}(Y_i,\hat{\mu}(Y)_i)}{\sigma^2}}=\int{\frac{\langle y-\mu,\hat{\mu}(y)\rangle}{\sigma^2} \psi(y;\mu,\sigma^2) dy}.
		\end{equation}
		If $\hat{\mu}$ is differentiable in Lebesgue almost all points and $\mathrm{div}(\hat{\mu})$ has finite first moment Stein's degrees of freedom of $\hat{\mu}$ is
		\begin{equation}
		\label{eq:df_s_def}
		\mathrm{df}_{\mathrm{S}}(\hat{\mu}) \coloneqq E(\mathrm{div} (\hat{\mu})(Y)).
		\end{equation}
	\end{defi}
	
	A simple expansion of the risk yields
	\begin{equation} \label{eq:riskexpan}
	\mathrm{Risk} =E\|Y-\hat{\mu}(Y)\|_2^2-n\sigma^2+2\sigma^2\mathrm{df}(\hat{\mu}).
	\end{equation}
	Hence $\|Y-\hat{\mu}(Y)\|_2^2-n\sigma^2+2\sigma^2\widehat{\mathrm{df}}$ is an unbiased risk estimate if $\widehat{\mathrm{df}}$ is an unbiased estimate of $\mathrm{df}(\hat{\mu})$. In practice, $\sigma^2$ must be estimated as well and a bias of $\widehat{\mathrm{df}}$ can also be preferable if it reduces the variance. Hence exact unbiasedness of a risk estimate based on \eqref{eq:riskexpan} is of secondary interest, but it is of interest to find adequate corrections of the squared error $\|Y-\hat{\mu}(Y)\|_2^2$ that can be used for model assessment and comparison. 
	
        If $\hat{\mu}$ is \textit{almost differentiable} then $\mathrm{df}(\hat{\mu})=\mathrm{df}_{\mathrm{S}}(\hat{\mu})$ due to Stein's lemma (Lemma 2 in \cite{stein1981}), in which case $\mathrm{div}(\hat{\mu})(Y)$ is an unbiased estimate of $\mathrm{df}(\hat{\mu})$. However, most estimators with discontinuities are not almost differentiable, and for such estimators it is not clear if $\mathrm{div}(\hat{\mu})(Y)$ is a useful estimate of the degrees of freedom. Indeed, our main result, Theorem \ref{theo:singular_repres}, provides a representation of $\mathrm{df}(\hat{\mu}) - \mathrm{df}_{\mathrm{S}}(\hat{\mu})$, which is nonzero for a range of estimators. The result provides the theoretical basis for establishing more adequate estimates of the degrees of freedom and thus the risk. Furthermore, Theorem \ref{theo:debiasedlasso} provides a quite remarkable connection between $\mathrm{df}(\hat{\mu})$ and $\mathrm{df}_{\mathrm{S}}(\hat{\mu})$ for the lasso-OLS estimator, which can be used to derive an estimate of $\mathrm{df}(\hat{\mu})$. This result is directly applicable in practice and provides fast and accurate risk estimation without the need for cross-validation, say.

	Our main result is derived under the assumptions on the estimator as stated below. To fix notation we let $B(x,r)$ denote the closed ball in $\mathbb{R}^n$ of radius $r$ and center $x$. Additionally, we let $\mathcal{H}^{n-1}$ denote the $n-1$ dimensional Hausdorff measure -- a generalisation of the surface measure of $n-1$ dimensional hypersurfaces in $\mathbb{R}^n$ (see e.g. \cite{evans1991measure} for details).
	
	\begin{ass}
		\label{ass:assumption1}
		The estimator $\hat{\mu}$ can be written as $\hat{\mu}=\sum_{i=1}^{N}1_{U_i}\hat{\mu}_i$ for a collection of open and disjoint sets $\{U_i\}_{i=1}^N$ with $\bigcup_{i=1}^N\overline{U}_i=\mathbb{R}^n$. Additionally, for each $i=1,...,N$:
		\begin{enumerate}[(a)]
			\item \label{itm:ass_1} The map $\hat{\mu}_i:\overline{U}_i\rightarrow\mathbb{R}^n$ is locally Lipschitz.
			\item \label{itm:ass_2} The random variable $1_{U_i}\emph{div}(\hat{\mu}_i)(Y)$ has finite first moment and $\|\hat{\mu}_i\|$ is polynomially bounded on $U_i$.
			\item \label{itm:ass_3} The function $r\mapsto \mathcal{H}^{n-1}\left(\partial U_i \cap B(0,r) \right)$ is polynomially bounded.
		\end{enumerate}
	\end{ass}

\begin{rem}
\label{rem:ass}
The following points are worth noting:
\begin{enumerate}[a)]
\item \label{itm:rem_ass_1} {\bf Boundary values of the estimator.} Assumption \ref{ass:assumption1}\eqref{itm:ass_3} implies that the boundaries of the sets $U_i$ are Lebesgue null sets, and thus that $\mathbb{R}^n \backslash \bigcup_i U_i$ has Lebesgue measure zero. The estimator $\hat{\mu}$ is here defined to be zero on this null set, but with $Y$ having an absolutely continuous distribution its value on a null set is irrelevant. Note, however, that Assumption \ref{ass:assumption1}\eqref{itm:ass_1} ensures that $\hat{\mu}_i$ is uniquely defined on $\partial U_i$. In a concrete case there may be a natural way to define $\hat{\mu}$ on the common boundary between $U_i$ and $U_j$, say, but we make no abstract attempt to select between $\mu_i$ and $\mu_j$ on the boundary.
\item \label{itm:rem_ass_2} {\bf Degrees of freedom.} Assumption \ref{ass:assumption1}\eqref{itm:ass_1} implies by Rademacher's theorem (Theorem 3.1.6 and 3.1.7 in \cite{federer1969geometric}) that $\mathrm{div}(\hat{\mu}_i)$ is defined Lebesgue a.e.. Combining this with Assumption \ref{ass:assumption1}\eqref{itm:ass_2} we conclude that under Assumption \ref{ass:assumption1} both $\mathrm{df}(\hat{\mu})$ and  $\mathrm{df}_{\mathrm{S}}(\hat{\mu})$ are well defined. 
\item \label{itm:rem_ass_3} {\bf Existence of normal vectors.} Assumption \ref{ass:assumption1}\eqref{itm:ass_3} implies that the sets $U_i$ have locally finite perimeter (see Theorem 5.11.1 in \cite{evans1991measure}), thus a measure theoretic outer unit normal $\eta_i$ is defined on a subset of $\partial U_i$. In fact, by Lemma \ref{lem:outer_unit_normal} Assumption \ref{ass:assumption1}\eqref{itm:ass_3} only needs to hold for the reduced boundary $\partial^*U_i$ (see Definition 5.7 and Lemma 5.8.1 in \cite{evans1991measure}). Whenever $\partial U_i$ is smooth the measure theoretic unit normal coincides with the usual pointwise unit normal. 
\end{enumerate}
\end{rem}
	
	Estimators that involve data driven variable selection will generally fulfil Assumption \ref{ass:assumption1} with each $U_i$ corresponding to a set of selected variables. Example \ref{exam:exam_of_assumption1} provides a thorough characterization of $U_i$ in the lasso-OLS setup. Moreover, a similar characterization of $U_i$ is given in Example \ref{exam:U_A} for a class of estimators defined via minimisation of a penalized loss function. 
	
	The conditions in Assumption \ref{ass:assumption1} are typically easy to verify, except perhaps the third condition, as it involves bounding Hausdorff measures. Appendix \ref{sec:semialg} provides some results that can be helpful for verifying the third condition. For estimators satisfying Assumption \ref{ass:assumption1} we have the following representation of the degrees of freedom. 
	
	\begin{theo}
		\label{theo:singular_repres}
		If $\hat{\mu}$ satisfies Assumption \ref{ass:assumption1} then 
		\begin{equation}
		\label{eq:df_struct}
		\mathrm{df}(\hat{\mu})=\mathrm{df}_{\mathrm{S}}(\hat{\mu}) + 
\frac{1}{2}\sum_{i\neq j}{  \int_{ \overline{U}_i\cap \overline{U}_j}{  \langle\hat{\mu}_j-\hat{\mu}_i , \eta_i\rangle\psi(\, \cdot \,; \mu, \sigma^2) \ d\mathcal{H}^{n-1} } },
		\end{equation}
		where $\eta_i$ denotes the measure theoretic outer unit normal to $\partial U_i$.
	\end{theo}

	The proof is in Appendix \ref{sec:proofs}. The essential part is an application of a generalized version of Gauss-Green's formula combined with a dominated convergence argument. Note that though $\overline{U}_i\cap \overline{U}_j$ is a Lebesgue null set -- on which $\hat{\mu}$ is defined to be zero -- $\hat{\mu}_j$ and $\hat{\mu}_i$ are uniquely defined by Assumption  \ref{ass:assumption1}\eqref{itm:ass_1} and generally non-zero and different, cf. also Remark  \ref{rem:ass}\eqref{itm:rem_ass_1}.
	
	If $\hat{\mu}$ satisfies Assumption \ref{ass:assumption1} and is continuous then \eqref{eq:df_struct} reduces to $\mathrm{df}(\hat{\mu})=\mathrm{df}_{\mathrm{S}}(\hat{\mu})$, which is Stein's lemma for a class of locally Lipschitz continuous estimators. The boundary integrals therefore account for potential jumps of $\hat{\mu}$ across the boundary of any two adjacent regions $U_i$ and $U_j$. For two-step procedures consisting of a model selection step followed by a parameter estimation step, $\mathrm{df}_{\mathrm{S}}$ generally only accounts for the contribution to the degrees of freedom by the estimation step, and the boundary integrals account for the contribution from the selection step.
	
	The following example illustrates how to verify Assumption \ref{ass:assumption1} for the lasso-OLS estimator, which is the estimator that will also be the main focus of the subsequent section. 
	
	\begin{exam}[The lasso-OLS estimator]
		\label{exam:exam_of_assumption1}
		Let $X$ be an $n\times p$-matrix. For any subset $A \subseteq \{1,...,p\}$, $X_A$ denotes the matrix whose columns are those of $X$ indexed by $A$, and similarly, $\beta_A\in \mathbb{R}^{|A|}$ denotes $(\beta_i)_{i\in A}$ for $\beta\in \mathbb{R}^p$.
		We let 
		\[ \mathcal{S}\coloneqq \left\{ S=\mathrm{col}(X_A) \mid A\subseteq \{1,..., p\} \right\} \]
		denote the set of subspaces spanned by columns of $X$. The orthogonal projection onto a subspace $S\in \mathcal{S}$ is denoted by $\Pi_S$.
		
		A \emph{lasso estimator} $\hat{\mu}_{\mathrm{lasso}}^{\lambda}(y)$ with tuning parameter $\lambda>0$ is defined as $\hat{\mu}_{\mathrm{lasso}}^{\lambda}(y)=X\hat{\beta}^{\lambda}$ where
		\[ \hat{\beta}^{\lambda} \in \underset{\beta}{\argmin} \ \frac{1}{2}\|y-X\beta\|^2_2+\lambda\|\beta\|_1.   \]
		We do not make any assumptions on $X$, and therefore it may happen that multiple $\hat{\beta}^{\lambda}$-solutions exist. For a solution $\hat{\beta}^{\lambda}$, the support, $\mathrm{supp}(\hat{\beta}^{\lambda}) \subseteq \{1,...,p\}$, is called an \emph{active set}. The lasso estimator $\hat{\mu}_{\mathrm{lasso}}^{\lambda}(y) = X \hat{\beta}^{\lambda}$ belongs to the space $\mathrm{col}(X_A)$ for $A = \mathrm{supp}(\hat{\beta}^{\lambda})$, and it follows by Lemma 7 in \cite{tibshirani2012} that there exists a Lebesgue null set $N$, such that $\mathrm{col}(X_A)$ is invariant with respect to the choice of the active set of solutions for $y \not \in N$. The map $\widehat{S}^{\lambda}:\mathbb{R}^n\setminus N\rightarrow \mathcal{S}$ returning $\mathrm{col}(X_A)$ when there is a solution $\hat{\beta}^\lambda$ with active set $A = \mathrm{supp}(\beta^{\lambda})$ is therefore well defined. The \emph{lasso-OLS} estimator $\hat{\mu}_{\textnormal{l-OLS}}^{\lambda} \coloneqq \Pi_{\widehat{S}^{\lambda}}$ is defined as the projection onto the space selected by the lasso, and is thus well-defined Lebesgue almost everywhere.
		
		By defining the disjoint selection events  
		\[ U_S^{\lambda} \coloneqq (\widehat{S}^\lambda = S) \]
		for each $S\in \mathcal{S}$, we immediately see from Lemma 6 in \cite{tibshirani2012} that each selection event is open and that $\mathbb{R}^n=\bigcup_{S\in \mathcal{S}} \overline{U}_S^{\lambda}$. We can safely ignore any empty $U_S^{\lambda}$. From the proof of Lemma 6 in \cite{tibshirani2012} we see that $\partial U_S^{\lambda}\subseteq (\bigcup_{T\in \mathcal{S}}U_T^{\lambda})^c$ is a finite union of affine subspaces of dimensions $\leq n-1$, and $r\mapsto \mathcal{H}^{n-1}(\partial U_S^{\lambda}\cap B(0,r))$ is thus polynomially bounded. This follows by elementary considerations, but it is also a consequence of Lemma \ref{lem:semialgebraic}. 	Consequently, 
		\[ \hat{\mu}_{\textnormal{l-OLS}}^{\lambda} = \sum_{S\in\mathcal{S}}{1_{U_S^{\lambda}} \Pi_S} \quad \text{almost everywhere,}\]
		 and it satisfies all conditions in Assumption \ref{ass:assumption1}. Figure \ref{fig:dlasso_decomp} provides an illustration of the partition of $\mathbb{R}^n$ for $n = p = 2$ for different choices of angles between the columns in $X$. 

Note that since $\hat{\mu}_{\textnormal{l-OLS}}^{\lambda} = \Pi_S$ on the open set $U_S^{\lambda}$, its divergence equals $\mathrm{dim}(S)$, hence Stein's degrees of freedom is
$$\mathrm{df}_S(\hat{\mu}_{\textnormal{l-OLS}}^{\lambda}) = E (\mathrm{dim}(\widehat{S}^{\lambda})).$$
From Lemma 3 in \cite{Tibshirani:2013} it follows that $\mathrm{dim}(\widehat{S}^{\lambda}) = |\mathrm{supp}(\hat{\beta}^\lambda)|$ whenever the columns of $X$ are in general position, which is useful for practical computations.

		\begin{figure}
				\centering
				\includegraphics[width=0.8\textwidth]{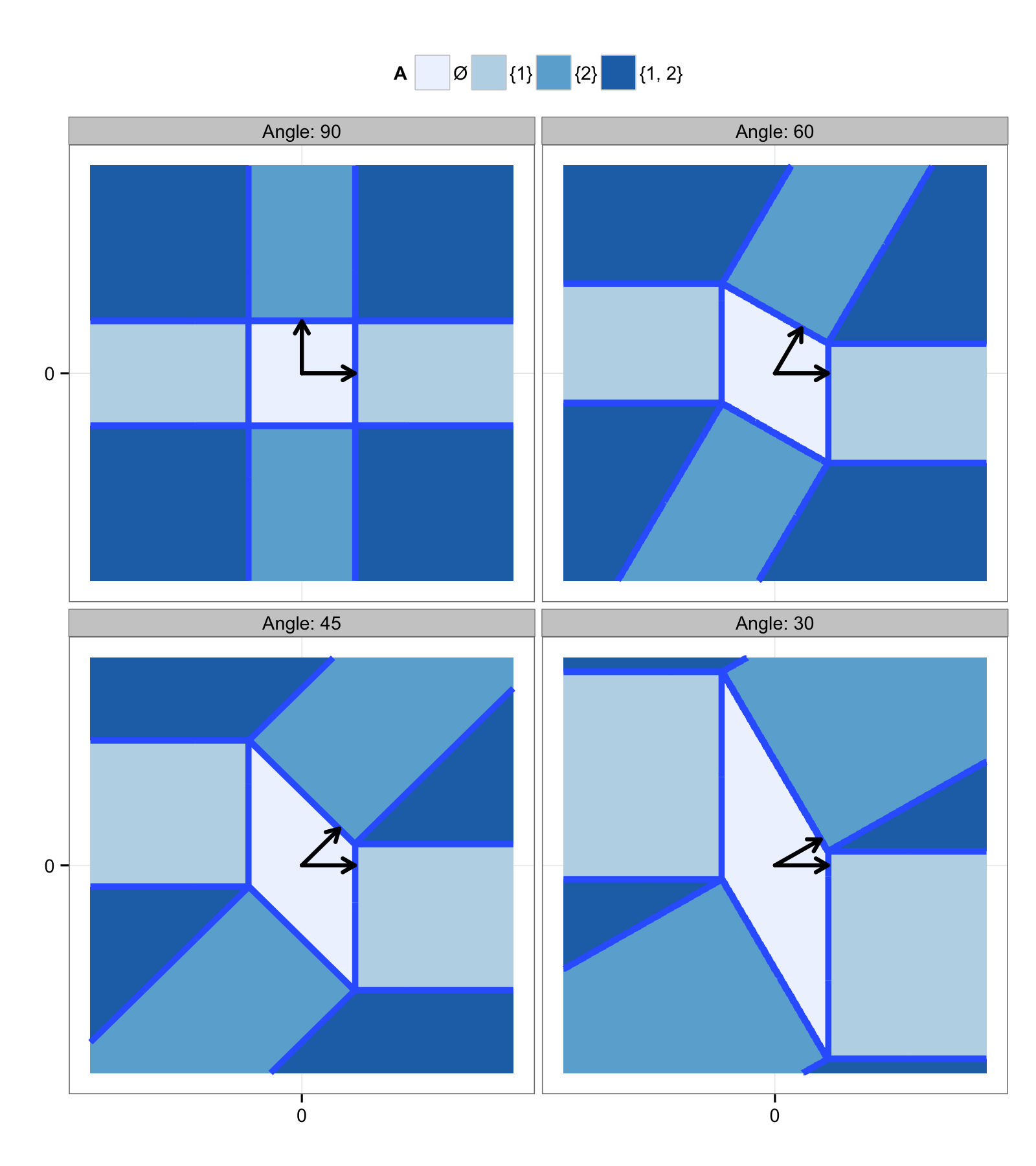}
				\caption{Illustrations of the decomposition of $\mathbb{R}^2$ into the four sets $U_{\emptyset}^1$, $U_{\{1\}}^1$, $U_{\{2\}}^1$ and $U_{\{1, 2\}}^1$ according to the lasso estimator with $\lambda = 1$. The set $U_{\emptyset}^1$ consists of the points shrunk to zero, the sets $U_{\{1\}}^1$ and $U_{\{2\}}^1$ to the points where either the second or the first coordinate, respectively, is shrunk to zero and $U_{\{1, 2\}}^1$ to the set where none of the coordinates are shrunk to zero. The decomposition depends on the angle between the two columns in $X$.}
				\label{fig:dlasso_decomp}
			\end{figure}

	 \QEDB
	\end{exam}

        The arguments above are based on results in \cite{tibshirani2012}, but see also \cite{Lee:2016b} for related characterizations 
of the selection events for lasso.

	\section{Risk estimation for lasso-OLS}
        \label{sec:l-OLS}
        
        It is not obvious how the general formula in Theorem \ref{theo:singular_repres} for $\mathrm{df}(\hat{\mu})$ can be used for computing or estimating the degrees of freedom. The first term of \eqref{eq:df_struct}, $\mathrm{df}_{\mathrm{S}}(\hat{\mu})$, may be estimated by $\mathrm{div}(\hat{\mu})(Y)$, but the second term is more difficult. In this section we show how this second term can be related to the derivative of $\lambda \mapsto \mathrm{df}_{\mathrm{S}}(\hat{\mu}^{\lambda}_{\textnormal{l-OLS}})$ 
for lasso-OLS. First we recapitulate the computations in \cite{Tibshirani:2015} of the degrees of freedom for lasso-OLS with $X$ orthogonal, which will reveal the general formula shown below.

	\begin{exam}[Continuation of Example \ref{exam:exam_of_assumption1}]
		\label{exam:diff}
		Assume that $n = p$ and $X=I$. In this case it is well known that the lasso and the lasso-OLS estimators become the soft and hard thresholding estimators, respectively. That is, 
		\[ \hat{\mu}_{\mathrm{lasso},i}^{\lambda}=\begin{cases}
		Y_i - \lambda \, \mathrm{sign}(Y_i) & \mathrm{if } \ |Y_i|> \lambda\\
		0 & \mathrm{otherwise}
		\end{cases} \qquad \mathrm{and} \qquad   \hat{\mu}_{\textnormal{l-OLS},i}^{\lambda}=\begin{cases}
                  Y_i  & \mathrm{if} \ |Y_i|> \lambda\\
		0 & \mathrm{otherwise}
		\end{cases}. \]
		We can write up closed form expressions for $\mathrm{df}(\hat{\mu}_{\textnormal{l-OLS}}^{\lambda})$ and $\mathrm{df}_{\mathrm{S}}(\hat{\mu}_{\textnormal{l-OLS}}^{\lambda})$:
		\begin{align*}
		\mathrm{df}_{\mathrm{S}}(\hat{\mu}_{\textnormal{l-OLS}}^{\lambda})&=\int{\psi(y;\mu,\sigma^2)\sum_{i}{1_{(|y_i|>\lambda)}} \ dy} = \sum_{i}{\int_{(|y_i|>\lambda)}{\psi(y_i;\mu_i,\sigma^2) \ dy_i }}\\
		&=\sum_{i}{ \Phi\left(\frac{-\lambda-\mu_i}{\sigma}\right) + \left(1 - \Phi\left(\frac{\lambda-\mu_i}{\sigma}\right)\right) },
		\end{align*}
		and as in \cite{Tibshirani:2015}
		\begin{align*}
		\mathrm{df}(\hat{\mu}_{\textnormal{l-OLS}}^{\lambda})& = \sum_{i}{\int_{\lambda}^{\infty}{\psi(y_i;\mu_i,\sigma^2)\frac{y_i(y_i-\mu_i)}{\sigma^2} \ dy_i } + \int_{-\infty}^{-\lambda}{\psi(y_i;\mu_i,\sigma^2)\frac{y_i(y_i-\mu_i)}{\sigma^2} \ dy_i }}\\
		&= \sum_{i}
		\left[ -\psi(y_i;\mu_i,\sigma^2)y_i \right]_{\lambda}^{\infty} + \int_{\lambda}^{\infty}{\psi(y_i;\mu_i,\sigma^2) \ dy_i } \\
                  & \hskip 10mm +	\left[ -\psi(y_i;\mu_i,\sigma^2)y_i \right]^{-\lambda}_{-\infty} +  \int_{-\infty}^{-\lambda}{\psi(y_i;\mu_i,\sigma^2) \ dy_i }
		  \\
		&= \lambda\sum_{i}{\left(\psi(\lambda;\mu_i,\sigma^2) + \psi(-\lambda;\mu_i,\sigma^2)\right)} + \mathrm{df}_{\mathrm{S}}(\hat{\mu}_{\textnormal{l-OLS}}^{\lambda}).
		\end{align*}	
		Letting $\partial_{\lambda}$ denote the differential operator with respect to $\lambda$ we observe that
		\begin{equation}
		\label{eq:df_relation_orthogonal}
		\mathrm{df}(\hat{\mu}_{\textnormal{l-OLS}}^{\lambda}) = \mathrm{df}_{\mathrm{S}}(\hat{\mu}_{\textnormal{l-OLS}}^{\lambda}) -\lambda\partial_{\lambda}\mathrm{df}_{\mathrm{S}}(\hat{\mu}_{\textnormal{l-OLS}}^{\lambda}),
		\end{equation} 
which is a striking identity. This is because the formula for $\mathrm{df}(\hat{\mu}_{\textnormal{l-OLS}}^{\lambda})$, though explicit, involves the unknown parameter $\mu$ and is not readily estimable. But we have the divergence estimator, $\sum_{i} 1_{(|y_i| > \lambda)}$, of $\mathrm{df}_{\mathrm{S}}(\hat{\mu}_{\textnormal{l-OLS}}^{\lambda})$, and if we from this can estimate its derivative as well, the formula above suggests how to estimate $\mathrm{df}(\hat{\mu}_{\textnormal{l-OLS}}^{\lambda})$.
\QEDB
	\end{exam}
	
	The remarkable fact that we will show is that \eqref{eq:df_relation_orthogonal} holds without the orthogonality assumption on $X$.

	\begin{theo}
		\label{theo:debiasedlasso}
		For the lasso-OLS estimator defined in Example \ref{exam:exam_of_assumption1} it holds that 
		\begin{equation}
		\label{eq:df_relation_lasso}
		\mathrm{df}(\hat{\mu}_{\textnormal{l-OLS}}^{\lambda}) = \mathrm{df}_{\mathrm{S}}(\hat{\mu}_{\textnormal{l-OLS}}^{\lambda}) - \lambda\partial_{\lambda}\mathrm{df}_{\mathrm{S}}(\hat{\mu}_{\textnormal{l-OLS}}^{\lambda})
		\end{equation}
where $\partial_{\lambda}$ denotes differentiation w.r.t. $\lambda$. 
	\end{theo}

Theorem \ref{theo:debiasedlasso} suggests that $\mathrm{df}(\hat{\mu}_{\textnormal{l-OLS}}^{\lambda})$ can be estimated by differentiation of an estimate of $\mathrm{df}_{\mathrm{S}}(\hat{\mu}_{\textnormal{l-OLS}}^{\lambda})$. The divergence estimate of Stein's degrees of freedom is, however, not differentiable as a function of $\lambda$, and we need to somehow smooth it. To this end it is convenient to reparametrise the penalization in terms of $\delta = \log(\lambda)$, so that with 
$$h(\delta) \coloneqq  \mathrm{df}_{\mathrm{S}}(\hat{\mu}_{\textnormal{l-OLS}}^{\exp(\delta)}),$$
then 
$$\mathrm{df}(\hat{\mu}_{\textnormal{l-OLS}}^{\exp(\delta)}) = h(\delta) - h'(\delta).$$ 
In simulations $h$ was found to be monotonically decreasing, and thus $h'$ to be negative, but we cannot prove that this is generally the case. The integral representation of $h'$ from Theorem \ref{theo:singular_repres} is not particularly helpful as the integrand can, in fact, be negative. Based on our computational observations -- and to reduce variance of the resulting estimate -- our proposal is based on the assumption that $h'$ is negative. It is effectively a kernel  smoother that estimates the intensity of jumps for a monotone jump process. 

We note that $\mathrm{dim}(\hat{S}^{\exp(\delta)})$
is an unbiased estimate of $h(\delta)$ and that the function $\delta \mapsto \mathrm{dim}(\hat{S}^{\exp(\delta)})$ is a step function. The problem of estimating the derivative, $h'$, of its mean is thus analogous to estimating the intensity for a jump process with one main difference; the step function can have jumps of negative as well as positive sign, though most jumps will be negative. Our proposed estimate ignores the positive excursions of the step function and is computed as follows:
\begin{itemize}
\item Compute the jump points, $\lambda_i$ and jump sizes, $\Delta_i \coloneqq \inf_{\lambda < \lambda_i}\mathrm{dim}(\hat{S}^{\lambda}) -  \mathrm{dim}(\hat{S}^{\lambda+})$, of the decreasing function $\lambda \mapsto \inf_{\lambda' < \lambda}\mathrm{dim}(\hat{S}^{\lambda'})$ for $i = 1, \ldots, M$.
\item Apply a kernel density smoother to the points $\delta_i = \log(\lambda_i)$ for $i = 1, \ldots, M$ counted with the multiplicities $\Delta_i$. In the simulations presented in this paper an adaptive Gaussian kernel density smoother was used (see Section 10.4.3.2 in \cite{Givens:2012}).
\item Rescale the density estimate by the total number of jumps, that is, by $\sum_{i=1}^M \Delta_i$.
\end{itemize}

As mentioned above, we can think of the proposed estimate of $h'$ as a non-parametric estimate of the intensity of the jumps for a monotonically decreasing jump process. Alternatively, we can think of it as smoothing the jumps by a sigmoidal function (the anti-derivative of the kernel) to obtain a smooth estimate of Stein's degrees of freedom, which can then be differentiated.  Note that even if $\Delta_i$ may always be 1 in theory, the jumps are in practice computed on a grid and may thus be larger than 1, which the procedure accounts for. The estimate of  $-\lambda\partial_{\lambda}\mathrm{df}_{\mathrm{S}}(\hat{\mu}_{\textnormal{l-OLS}}^{\lambda})$ resulting from the procedure above is denoted by $\widehat{\partial}$.

	\begin{figure}
		\centering
		\includegraphics[width = \textwidth]{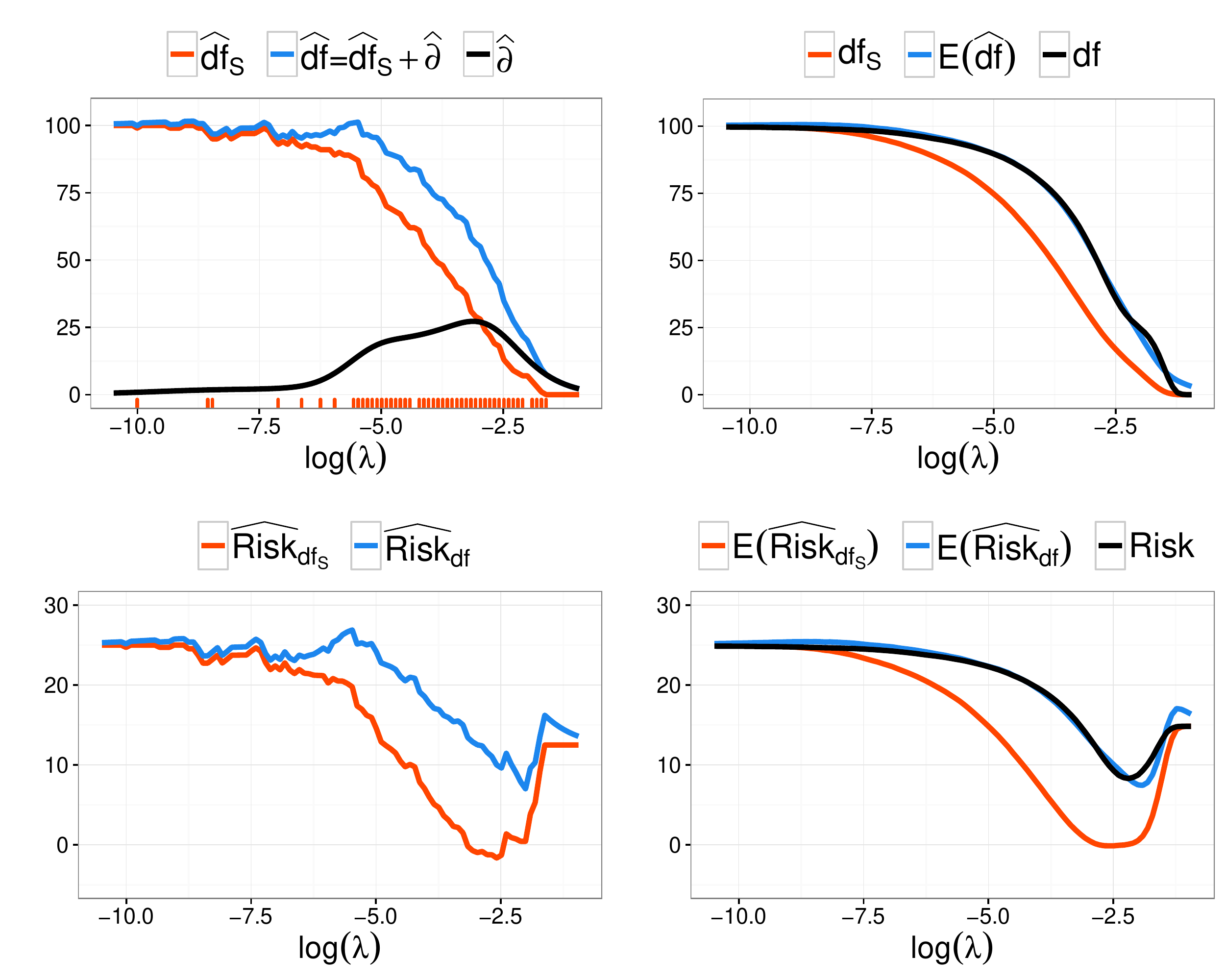}
		\caption{Left: Realization of the estimates of degrees of freedom $\hat{\mathrm{df}}_S = \mathrm{dim}(\hat{S}^{\lambda})$ and $\hat{\mathrm{df}} = \mathrm{dim}(\hat{S}^{\lambda}) + \widehat{\partial}$ as well as the correction term $\widehat{\partial}$ as a function of $\log(\lambda)$ (top) and corresponding estimates of the risk (bottom). Right: Similar to the left but mean values of the estimates obtained by averaging over 1000 samples along with the degrees of freedom $\mathrm{df} = \mathrm{df}(\hat{\mu}_{\textnormal{l-OLS}}^{\lambda})$ obtained from the 1000 samples using the covariance definition \eqref{eq:df_def}. The design parameters were: $\sigma=0.5$, $n=p=100$, $\gamma =1$, $\alpha =0.1$ and the design type was (S) with constant correlation of $\rho = 0.1$ (see Section \ref{sec:sim}).}
		\label{fig:dlasso}
	\end{figure}

Using $\mathrm{dim}(\hat{S}^{\lambda}) +\widehat{\partial}$ as an estimate of degrees of freedom leads to the risk estimate
\begin{equation}
\label{eq:riskhat}
\widehat{\mathrm{Risk}}_{\mathrm{df}} \coloneqq \|Y-\hat{\mu}_{\textnormal{l-OLS}}^{\lambda}\|_2^2-n\sigma^2+2\sigma^2\left(\mathrm{dim}(\hat{S}^{\lambda}) +\widehat{\partial}\right).
\end{equation}
For an example of the above estimate see Figure \ref{fig:dlasso}, where $\widehat{\partial}$ and   $\widehat{\mathrm{Risk}}_{\mathrm{df}}$ are applied to a single realization of $Y$ along with an average over 1000 replications.

To prove Theorem \ref{theo:debiasedlasso} we prove a more general intermediate result for estimators that are parametrised in a similar way by a tuning parameter. We use in the following $D$ to denote the differential operator w.r.t. $y$.

	\begin{prop}
		\label{theo:scaling}
		Let $q>0$ and suppose that $\hat{\mu}^{\lambda}=\sum_{i}{1_{U_i^{\lambda}}}\hat{\mu}_i$ where
		\begin{equation}
		\label{eq:set_cond}
		U_i^{\lambda} = \lambda^qU_i^1, \qquad \text{for all } i=1,...,N.
		\end{equation}
		Assume that $\mathrm{div}( \hat{\mu}_i)$ is locally Lipschitz and both  $\mathrm{div}( \hat{\mu}_i)$ and $D(\mathrm{div}(\hat{\mu}_i))$ are polynomially bounded for each $i=1,...,N$. If $\hat{\mu}^{1}$ satisfies Assumption \ref{ass:assumption1} then
		\begin{equation}
		\label{eq:res}
		-\frac{\lambda}{q}\partial_{\lambda}\mathrm{df}_{\mathrm{S}}(\hat{\mu}^{\lambda})=\frac{1}{2}\sum_{i\neq j}{ \int_{\overline{U}_i^{\lambda} \cap \overline{U}_j^{\lambda}}{ \Big(\mathrm{div}(\hat{\mu}_j)(y)- \mathrm{div}(\hat{\mu}_i)(y) \Big) \langle y , \eta_i\rangle \psi(y; \mu,\sigma^2)  \ d\mathcal{H}^{n-1}(y)}}.
		\end{equation}
	\end{prop}
	\begin{proof} 
		First observe that $\partial U_i^{\lambda}\cap B(0,r) = \lambda^q (\partial U_i^{1}\cap B(0,r/\lambda^q))$, hence if $\hat{\mu}^{1}$ satisfies Assumption \ref{ass:assumption1} so does $\hat{\mu}^{\lambda}$ for all $\lambda$. Next, the change of variable formula yields
		\begin{align*}
		\mathrm{df}_{\mathrm{S}}(\hat{\mu}^{\lambda}) & =\int{\psi(y) \mathrm{div} (\hat{\mu}^{\lambda})(y) \, dy} 
 =\sum_{i}{\int_{U_i^{\lambda}}{ \psi(y) \mathrm{div} (\hat{\mu}_i)(y) \, dy }} \\ 
& =\sum_{i}{\int_{U_i^{1}}{ \lambda^{qn}\left(\psi\mathrm{div}( \hat{\mu}_i)\right)(\lambda^qz) \ dz }}.
		\end{align*}
Here $\psi = \psi(\cdot; \mu, \sigma^2)$ to ease notation. 

		The last integrand is differentiable w.r.t. $\lambda$ (for Lebesgue a.a. $z$) and its derivative is
		\begin{align*}
		qn\lambda^{qn-1}&\left(\psi\mathrm{div}(\hat{\mu}_i)\right)(\lambda^qz)+\lambda^{qn}\left\langle D\left(\psi\mathrm{div} (\hat{\mu}_i)\right)(\lambda^qz), q\lambda^{q-1}z \right\rangle \\
		&=\frac{q}{\lambda}\lambda^{qn}\left(n\left(\psi\mathrm{div} (\hat{\mu}_i)\right)(\lambda^qz) +   \left\langle D\left(\psi\mathrm{div}(\hat{\mu}_i)\right)(\lambda^qz), \lambda^qz\right\rangle \right),
		\end{align*}
		which is dominated in a neighbourhood of $\lambda$ by an integrable function due to the polynomial bounds. Hence, by the change of variable formula 
		\begin{align*}
		\frac{\lambda}{q}\partial_{\lambda}\mathrm{df}_{\mathrm{S}}(\hat{\mu}^{\lambda})&=\sum_{i}{ \int_{U_i^{1}}{ \lambda^{qn}\left(n\left(\psi\mathrm{div} (\hat{\mu}_i)\right)(\lambda^qz) +   \left\langle D\left(\psi\mathrm{div}(\hat{\mu}_i)\right)(\lambda^qz), \lambda^qz\right\rangle \right)  dz }}\\
		&=\sum_{i}{ \int_{U_i^{\lambda}}{ n \left(\psi\mathrm{div} (\hat{\mu}_i)\right)(y) +   \left\langle D\left(\psi\mathrm{div}(\hat{\mu}_i)\right)(y), y\right\rangle  dy }}\\
		&=\sum_{i}{ \int_{U_i^{\lambda}}{ n \left(\psi\mathrm{div} (\hat{\mu}_i)\right)(y)  +   \left\langle \left(\psi D\mathrm{div}(\hat{\mu}_i) + \mathrm{div}(\hat{\mu}_i)D\psi \right)(y), y\right\rangle \ dy }} \\
		&=\sum_{i}{ \int_{U_i^{\lambda}}{ \psi(y) \mathrm{div}\left(y  \mathrm{div}(\hat{\mu}_i)(y)\right) + \left\langle D\psi(y), y\mathrm{div}(\hat{\mu}_i)(y)\right\rangle dy }}.
		\end{align*}
		The last line is identified as $\mathrm{df}_{\mathrm{S}}(\tilde{\mu}^{\lambda})-\mathrm{df}(\tilde{\mu}^{\lambda})$, where 
$$\tilde{\mu}^{\lambda}(y):=\sum_{i}{1_{U_i^{\lambda}}(y)y\mathrm{div} (\hat{\mu}_i)(y)}.$$ Finally \eqref{eq:res} follows by applying Theorem \ref{theo:singular_repres} to $\tilde{\mu}^{\lambda}$ (which also satisfies Assumption \ref{ass:assumption1}). 
	\end{proof}	
	
	\begin{exam}
		\label{exam:U_A}
		There are naturally occurring examples besides the lasso selection sets that satisfy \eqref{eq:set_cond}. Consider still a linear regression setup with $X$ an $n\times p$-matrix. Let $\ell$ denote the penalized loss function
		\[  \ell(y,\beta,\lambda)=\frac{1}{2}\left\|y-X\beta\right\|_2^2+\lambda \mathrm{Pen}(\beta),  \]
		for some penalty function $\mathrm{Pen}:\mathbb{R}^p\rightarrow\mathbb{R}$ and define the sets 
		\begin{equation}
		\label{eq:U_A} 
		U_A^{\lambda}  = \mathrm{int} \left\{ y\in \mathbb{R}^n \ \bigg\vert \ \underset{\beta:\mathrm{supp}(\beta) = A}{\inf}\ell(y,\beta,\lambda) = \underset{\beta}{\inf}\ \ell(y,\beta,\lambda) \right\}, 
		\end{equation}
		for each $A\subseteq \{1,...,p\}$. Hence any $y\in U_A^{\lambda}$ has $A$ as an active set. If $\mathrm{Pen}$ is \emph{positive homogeneous} of degree $k \in [0, 2)$ then
		\[ \ell\left(\lambda^{\frac{1}{2-k}}y,\lambda^{\frac{1}{2-k}}\beta,\lambda \right)=\lambda^{\frac{2}{2-k}}\ell\left(y,\beta,1\right). \] 
		Hence $U^{\lambda}_A=\lambda^{\frac{1}{2-k}}U_A^1$ holds for all $A\subseteq \{1,...,p\}$ and $\lambda>0$. The (quasi) norms, $\mathrm{Pen}(\beta)=\|\beta\|_k^k$ for $k \in (0,2)$, and $\mathrm{Pen}(\beta)=\|\beta\|_0 = |\mathrm{supp}(\beta)|$ are examples of positive homogeneous penalties. For these penalties only $k \in [0,1]$ will result in variable selection. With $\mathrm{Pen}(\cdot) = \| \cdot\|_1$ we see that for lasso the sets $U_S^{\lambda}$ in \ref{exam:exam_of_assumption1} satisfy \eqref{eq:set_cond} with $q = 1$. 

\QEDB
	\end{exam}

	\begin{proof}[Proof of Theorem \ref{theo:debiasedlasso}]
		Let $(U_S^{\lambda})_{S\in\mathcal{S}}$ be defined as in Example \ref{exam:exam_of_assumption1}, where it was also shown that Assumption \ref{ass:assumption1} holds for the lasso-OLS estimator. Moreover, from Example \ref{exam:U_A} we see that $U_S^{\lambda}=\lambda U_S^1$ for all $\lambda>0$ and $S\in \mathcal{S}$. By Theorem \ref{theo:singular_repres} we know that the left hand side of \eqref{eq:df_relation_lasso} is
		\begin{equation}
                \begin{aligned}
		\mathrm{df}(\hat{\mu}_{\textnormal{l-OLS}}^{\lambda}) & - \mathrm{df}_{\mathrm{S}}(\hat{\mu}_{\textnormal{l-OLS}}^{\lambda}) \\
                & = \frac{1}{2}\sum_{S_1\neq S_2}\int_{\overline{U}^{\lambda}_{S_1}\cap  \overline{U}^{\lambda}_{S_2}}{\langle (\Pi_{S_2}-\Pi_{S_1})y, \eta_{S_1}(y)\rangle \psi(y)\ d\mathcal{H}^{n-1}(y) }.
                \end{aligned}
		\label{eq:lhs}
		\end{equation}

        It will first be established that $\overline{U}^{\lambda}_{S_1}\cap  \overline{U}^{\lambda}_{S_2}$ for $S_1\neq S_2$ is a $\mathcal{H}^{n-1}$ null set unless $S_1$ and $S_2$ are nested and their dimensions differ by one.

	By definition $\hat{\mu}_{\mathrm{lasso}}^\lambda\in S$ on $U_S^{\lambda}$, and by continuity of $\hat{\mu}_{\mathrm{lasso}}^{\lambda}$ (a consequence of Lemma 3 in \cite{tibshirani2012}) we conclude that the same is true on  $\overline{U}_S^{\lambda}$. Hence for $S_1,S_2\in \mathcal{S}$
		\begin{equation}
		\label{eq:lasso_fit}
		\hat{\mu}_{\mathrm{lasso}}^\lambda\in S_1\cap S_2 \text{ on } \overline{U}_{S_1}^{\lambda}\cap \overline{U}_{S_2}^{\lambda}.
		\end{equation}
		For $A\subseteq\{1,...,p\}$ and $s\in \{-1,1\}^{|A|}$ we define the set 
$$L_{A,s}\coloneqq \{u \in \mathbb{R}^n \mid X_A^Tu=\lambda s \}.$$ 
It now follows from the first order subgradient conditions for lasso that
		\begin{equation}
		\label{eq:lasso_resid_cone}
		y-\hat{\mu}_{\mathrm{lasso}}^\lambda\in \bigcup_{\substack{A\subseteq\{1,...,p\}:\\ \mathrm{col}(X_A) = S}} \bigcup_{s \in \{-1,1\}^{|A|}}L_{A,s}
		\end{equation}
		for all $y\in U_S^{\lambda}$. Note that the dimension of the above set is $n - \mathrm{dim}(S)$. Since the set is closed and $\hat{\mu}_{\mathrm{lasso}}^\lambda$ is continuous, \eqref{eq:lasso_resid_cone} holds for $y\in \overline{U}_S^{\lambda}$ as well. We therefore conclude that 
		\begin{equation}
				\label{eq:lasso_resid}
				\begin{aligned}
				y - \hat{\mu}_{\mathrm{lasso}}^\lambda & \in \left(\bigcup_{\substack{A\subseteq\{1,...,p\}:\\ \mathrm{col}(X_A) = S_1}} \bigcup_{s \in \{-1,1\}^{|A|}}L_{A,s}\right) \cap \left(\bigcup_{\substack{A\subseteq\{1,...,p\}:\\ \mathrm{col}(X_A) = S_2}} \bigcup_{s \in \{-1,1\}^{|A|}}L_{A,s}\right)\\
				& \subseteq\bigcup_{\substack{A\subseteq\{1,...,p\}:\\ \mathrm{col}(X_A) = S_1 + S_2}} \bigcup_{s \in \{-1,1\}^{|A|}}L_{A,s}
				\end{aligned}
		\end{equation}
		for all $y\in \overline{U}_{S_1}^{\lambda}\cap \overline{U}_{S_2}^{\lambda}$ and $S_1,S_2\in \mathcal{S}$.		
		
		From \eqref{eq:lasso_fit} and \eqref{eq:lasso_resid} we deduce that
		\begin{equation}
		\label{eq:U_S_lasso}
		\overline{U}_{S_1}^{\lambda}\cap \overline{U}_{S_2}^{\lambda} \subseteq S_1\cap S_2  +  \bigcup_{\substack{A\subseteq\{1,...,p\}:\\ \mathrm{col}(X_A) = S_1 + S_2}} \bigcup_{s \in \{-1,1\}^{|A|}}L_{A,s}
		\end{equation}
		for $S_1,S_2\in \mathcal{S}$. Consequently, if $S_1\neq S_2$ then $\mathcal{H}^{n-1}\left(\overline{U}_{S_1}^{\lambda}\cap \overline{U}_{S_2}^{\lambda}\right)=0$, unless $S_1$ and $S_2$ are nested and their dimensions differ by $1$.

We can therefore assume $S_1\subseteq S_2$ and $\mathrm{dim}(S_2)=\mathrm{dim}(S_1)+1$. Furthermore, $S_2 \ominus S_1 = (S_1 + S_2) \ominus (S_1\cap S_2)$ is orthogonal to any of the faces $S_1\cap S_2 + L_{A,s}$ in \eqref{eq:U_S_lasso} and thus also orthogonal to $\overline{U}^{\lambda}_{S_1}\cap  \overline{U}^{\lambda}_{S_2}$. This implies that $\eta_{S_1} = (\Pi_{S_2}-\Pi_{S_1})\eta_{S_1}$ and hence \eqref{eq:lhs} becomes
		\begin{align*}
		\mathrm{df}&(\hat{\mu}_{\textnormal{l-OLS}}^{\lambda})  - \mathrm{df}_{\mathrm{S}}(\hat{\mu}_{\textnormal{l-OLS}}^{\lambda})\\
		&= \sum_{\substack{S_1\subseteq S_2,\\
				\mathrm{dim}(S_2)=\mathrm{dim}(S_1)+1}}\int_{\overline{U}^{\lambda}_{S_1}\cap  \overline{U}^{\lambda}_{S_2}}{\langle y, \eta_{S_1}(y)\rangle \psi(y)\ d\mathcal{H}^{n-1}(y)}\\
		&= \sum_{\substack{S_1\subseteq S_2,\\	\mathrm{dim}(S_2)=\mathrm{dim}(S_1)+1}}\int_{\overline{U}^{\lambda}_{S_1}\cap  \overline{U}^{\lambda}_{S_2}} \underbrace{\left[\mathrm{div}(\Pi_{S_2}) - \mathrm{div}(\Pi_{S_1})\right]}_{=\mathrm{dim}(S_2)-\mathrm{dim}(S_1) = 1}\langle y, \eta_{S_1}(y) \rangle \psi(y)\ d\mathcal{H}^{n-1}(y)\\
		&=-\lambda\partial_{\lambda}\mathrm{df}_{\mathrm{S}}(\hat{\mu}_{\textnormal{l-OLS}}^{\lambda})
		\end{align*}
		by Proposition \ref{theo:scaling}.	
	\end{proof}

	\section{Simulation Study}
        \label{sec:sim}

        We report in this section the results from an extensive simulation study, whose purpose was to quantify how $\widehat{\mathrm{Risk}}_{\mathrm{df}}$ given by \eqref{eq:riskhat} performs as an estimate of the risk and in terms of selecting the penalty parameter $\lambda$. Its performance was compared to alternatives for risk estimation and tuning, and the resulting lasso-OLS estimator was compared to the lasso estimator. Throughout, the R package \textit{glmnet}, \cite{friedman2010}, was used to compute the lasso solution path. This section is divided into subsections describing estimators and risk estimates, the design of the simulation study, and the results of the simulation study.

\subsection{Estimators and risk estimates}

The first alternative risk estimate for lasso-OLS is
    \begin{equation}
	\label{eq:risk_est}
	\widehat{\mathrm{Risk}}_{\mathrm{df}_{\mathrm{S}}}=\|Y-\hat{\mu}_{\textnormal{l-OLS}}^{\lambda}\|_2^2-n\sigma^2+2\sigma^2\mathrm{dim}(\hat{S}^{\lambda}),
	\end{equation}
	which does not adjust for the variable selection performed by lasso-OLS. The second alternative is $K$-fold cross-validation (denoted $\widehat{\mathrm{Risk}}_{\textnormal{CV-K}}$) with $K=5,10$. This risk estimate is given by
        \begin{equation}
        \label{eq:risk_cv}
        \widehat{\mathrm{Risk}}_{\textnormal{CV-K}}\coloneqq \sum_{k=1}^{K}{\| Y_k-X_k\hat{\beta}_{\textnormal{l-OLS}}^{\lambda}(Y_{-k},X_{-k}) \|^2_2} - n\sigma^2, 
        \end{equation}
        where $Y_k$ and $X_k$ denote the entries of $Y$ and rows of $X$, respectively, corresponding to the $k$th fold, and similarly, $Y_{-k}$ and $X_{-k}$ denote the entries and rows not in the $k$th fold. 

The lasso estimator was tuned by minimising the risk estimate 
	\begin{equation}
		\label{eq:risk_est_lasso}
		\widehat{\mathrm{Risk}}_{\mathrm{lasso}}=\|Y-\hat{\mu}_{\mathrm{lasso}}^{\lambda}\|_2^2-n\sigma^2+2\sigma^2\mathrm{dim}(\hat{S}^{\lambda}).
		\end{equation}

For $\mathrm{tuning}\in \{\mathrm{df}, \mathrm{df}_{\mathrm{S}}, \textnormal{CV-5}, \textnormal{CV-10},\mathrm{lasso} \}$ we let $\hat{\lambda}_{\mathrm{tuning}}$ denote the value of $\lambda$ that minimises $\widehat{\mathrm{Risk}}_{\mathrm{tuning}}$. The risk of the resulting estimator is denoted
$$\mathrm{Risk}(\mathrm{tuning}) \coloneqq E \|\mu - \hat{\mu}_{\textnormal{l-OLS}}^{\hat{\lambda}_{\mathrm{tuning}}}\|_2^2$$
for all but the $\mathrm{lasso}$-tuning, whose risk instead is 
$$\mathrm{Risk}(\mathrm{lasso}) \coloneqq E \|\mu - \hat{\mu}_{\mathrm{lasso}}^{\hat{\lambda}_{\mathrm{lasso}}}\|_2^2.$$

When the true mean is $\mu = X \beta$ with $\mathrm{supp}(\beta) = A$ we refer to $\Pi_A$ as the oracle-OLS estimator. This usage of the oracle terminology is in accordance with e.g. \cite{fan2014}. Its risk is 
$$E\|\mu - \Pi_A Y\|_2^2 = \sigma^2 \mathrm{rank}(X_A).$$
The results from the simulation study are reported in terms of $\mathrm{Risk}(\mathrm{tuning}) / (\sigma^2 n)$ for each tuning method, which can then be compared to $\mathrm{rank}(X_A) / n$ -- the fraction of nonzero parameters.

All simulations were carried out assuming either that $\sigma^2$ was known or using the following estimator of $\sigma^2$: first the lasso path $\lambda \mapsto \hat{\mu}_{\mathrm{lasso}}(\lambda)$ was calculated, then $\hat{\lambda}$ was selected by minimising the generalized cross-validation criterion
$$ \mathrm{gcv}(\lambda) = \frac{\|Y - \hat{\mu}_{\mathrm{lasso}}^\lambda\|_2^2}{\left(1 - \frac{\mathrm{dim}(\hat{S}^{\lambda})}{n}\right)^2}, $$
and $\sigma^2$ was finally estimated as 
$$\hat{\sigma}^2 = \frac{\|Y - \hat{\mu}^{\hat{\lambda}}_{\mathrm{lasso}}\|^2_2}{n - \mathrm{dim}(\hat{S}^{\lambda})}.$$ 
The main reason for choosing this estimator was computational efficiency, as the lasso path must be calculated for lasso-OLS anyway. Thus this variance estimate has virtually no extra computational costs. See also \cite{Reid:2015} for a comprehensive comparison of variance estimators.

\subsection{Simulation study design} \label{sec:simdesign}

In the simulation study the mean was given as $X \beta$ with 
        $$ \beta_i=\begin{cases}
	\gamma^{i-1} & \text{if } i\leq \lceil n \alpha \rceil\\
	0 & \mathrm{otherwise}
	\end{cases}$$
        for different choices of the dimension $n$, the $n \times p$ design matrix $X$ and the parameters $\gamma$ and $\alpha$. 

Two simulation designs were implemented with parameters as follows:
\begin{center}
\begin{tabular}{ccc}
\begin{tabular}{c}
Parameter \\
\hline
$\sigma$  \\
$\alpha$  \\
$n$  \\
$p$ \\
$\gamma$ \\
$X$ \\
$\rho$
\end{tabular}
& 
\begin{tabular}{ccccc}
\multicolumn{5}{c}{Values for simulation study I} \\
\hline
0.5 \\
0.1 \\
50 & 100 & 200 & 400 & 800  \\
200 & 2000 & 20000 \\
1 \\
S  \\
0.1  
\end{tabular}
& 
\begin{tabular}{ccccc}
\multicolumn{5}{c}{Values for simulation study II} \\
\hline
0.1 & 0.2 & 0.5 & 1 & 2  \\
0 & 0.05 & 0.1 & 0.3 & 0.5 \\
100 & 200 & & & \\
$n$ \\
1 & 0.9 & & & \\
O & S & E & & \\
0 & 0.1 & 0.4 & 0.7 \\
\end{tabular} 
\end{tabular}
\end{center}

The parameter $\rho$ and the values of the design require some explanation. The three different design types are:
\begin{itemize}
\item Orthogonal (O), where $X = I$. 
\item Simulated (S), where the columns of $X$ are standard normally distributed with one of the following correlation structures:
\begin{itemize}
 \item Autoregressive setup: $\mathrm{corr}(X_i, X_j) = \rho^{|i-j|}$ for all $i\neq j$.
 \item Constant correlation setup: $\mathrm{corr}(X_i, X_j) = \rho$ for all $i\neq j$.
\end{itemize}
\item Empirical (E), where the rows and columns are randomly selected from the $240\times377$ matrix of microRNA expression values as used in the earlier study by \cite{vincent2014}. 
\end{itemize}
The columns of the simulated and empirical designs were standardized to have norm one to obtain a comparable signal-to-noise ratio across the three designs.

The risk estimates were based on 1000 samples for each combination of the parameters, which were generated as follows. For each of the 1000 samples a design matrix $X$ was created/simulated and a single realization of $Y\sim\mathcal{N}(X\beta,\sigma^2I_n)$ was drawn. For each sample the losses $\|\mu - \hat{\mu}_{\mathrm{lasso}}^{\hat{\lambda}_{\mathrm{lasso}}}\|^2_2$  and  $\|\mu - \hat{\mu}_{\textnormal{l-OLS}}^{\hat{\lambda}_{\mathrm{tuning}}}\|^2_2$ for the different tuning methods were computed. The risks were estimated as the average of the losses over the 1000 samples.

In order to assess robustness to deviations from the Gaussian noise assumption, we replicated the second study design with two types of non-Gaussian noise: a $t$-distribution with 3 degrees of freedom, and a skew normal distribution with shape parameter 3. Location and scale parameters were set so that the noise distribution had mean 0 and variance $\sigma^2$.

\begin{figure}
    	\centering
    	\includegraphics[width=\textwidth]{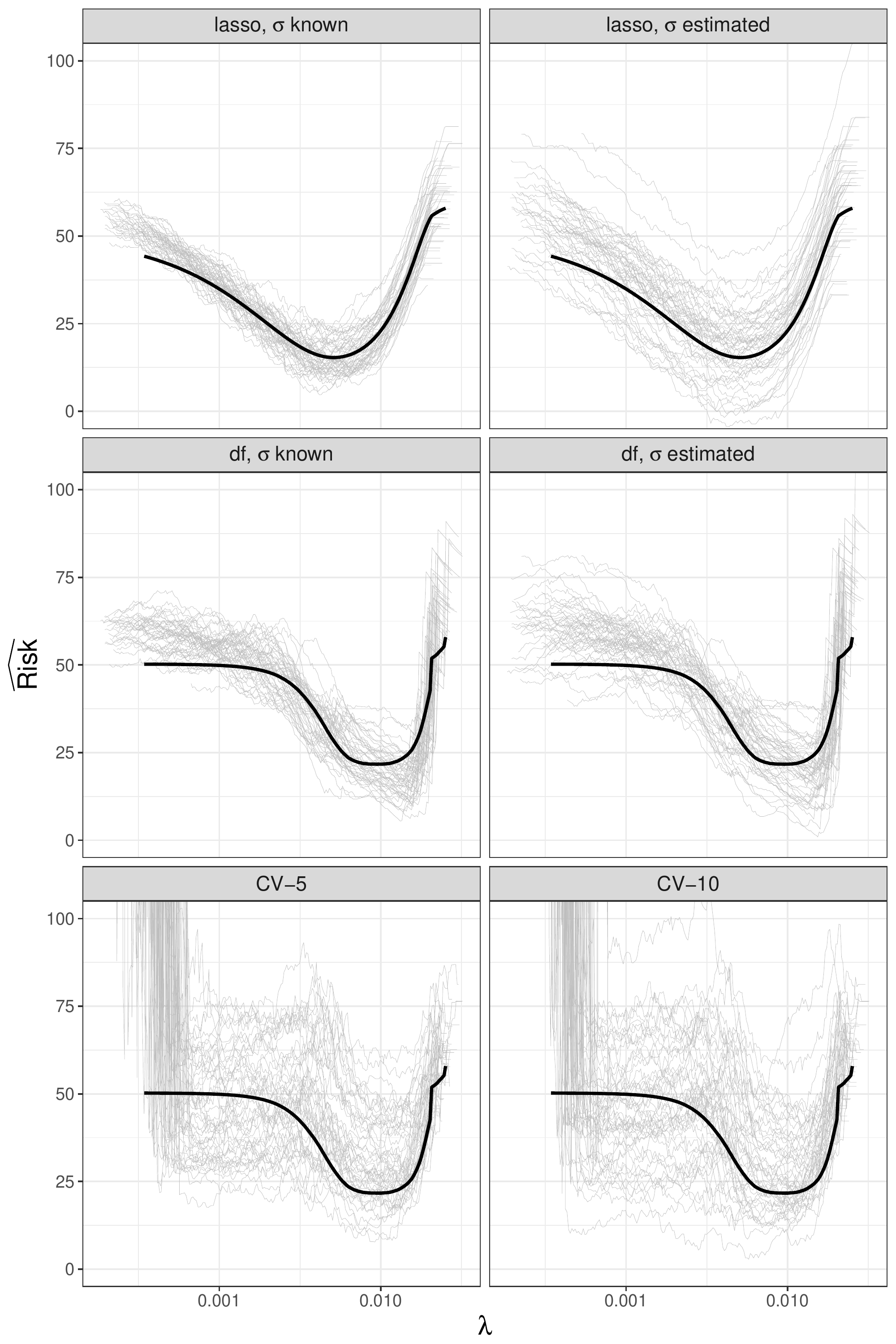}
    	\caption{Risk estimates $\widehat{\mathrm{Risk}}_{\mathrm{df}}$, $\widehat{\mathrm{Risk}}_{\textnormal{CV-5}}$, $\widehat{\mathrm{Risk}}_{\textnormal{CV-10}}$ and $\widehat{\mathrm{Risk}}_{\mathrm{lasso}}$ (gray lines) for 50 samples as a function of $\lambda$. The black lines are Monte Carlo estimates of the true risks. The design parameters were: $n=200$, $p=2000$, $\sigma=0.5$, $\gamma =1$, $\alpha =0.1$,  and the design type was (S) with a constant correlation of $\rho = 0.1$ (see Section \ref{sec:simdesign}).}
    	\label{fig:risk_cloud}
\end{figure}

\begin{figure}
    	\centering
    	\includegraphics[width=\textwidth]{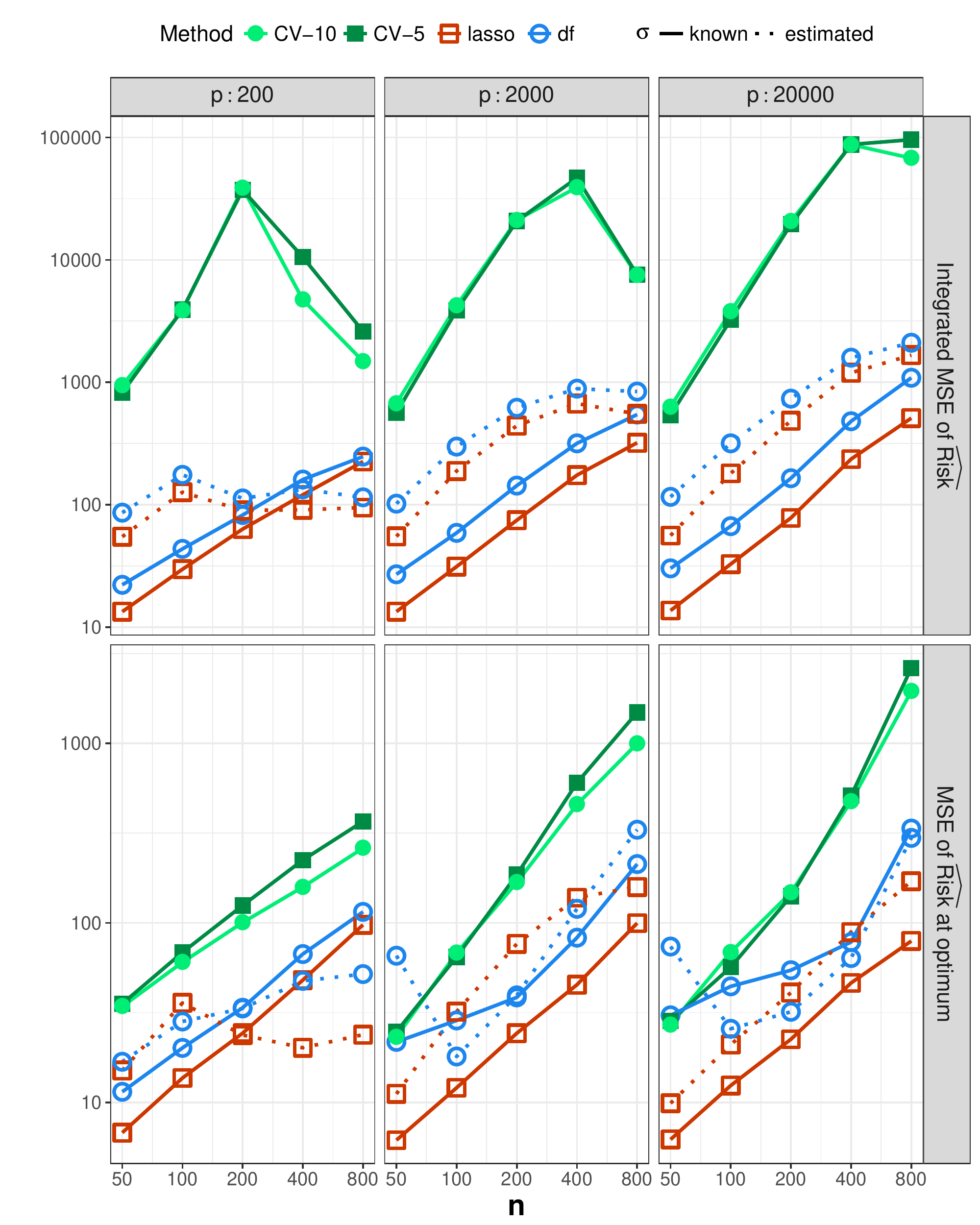}
    	\caption{Integrated mean squared error (top) and mean squared error at the optimal value of $\lambda$, $\hat{\lambda}$ (bottom) of the risk estimates  $\widehat{\mathrm{Risk}}_{\mathrm{df}}$, $\widehat{\mathrm{Risk}}_{\textnormal{CV-5}}$, $\widehat{\mathrm{Risk}}_{\textnormal{CV-10}}$ and $\widehat{\mathrm{Risk}}_{\mathrm{lasso}}$. The integrated mean squared error was computed over the interval $[\hat{\lambda}/10,10\hat{\lambda} ]$ of $\log(\lambda)$-values. The design parameters were: $\sigma=0.5$, $\gamma =1$, $\alpha =0.1$,  and the design type was (S) with a constant correlation of $\rho = 0.1$ (see Section  \ref{sec:simdesign})}
    	\label{fig:MSEs}
\end{figure}	

\subsection{Results from study I}

We first report on the accuracy of the risk estimates. Figure \ref{fig:risk_cloud} shows the risk estimates as a function of $\lambda$ for 50 samples along with a Monte Carlo estimate of the true risk. Cross-validation appears to give more variable estimates of the risk than $\widehat{\mathrm{Risk}}_{\textrm{df}}$ across the entire range of $\lambda$-values. This is true even when the variance is estimated, though estimation of the variance does appear to degrade the performance of the risk estimates. We note that $\widehat{\mathrm{Risk}}_{\textrm{df}}$ does not appear to be much more variable than $\widehat{\mathrm{Risk}}_{\mathrm{lasso}}$, though the former relies on the additional smoothed term for the estimation of degrees of freedom. 

Figure \ref{fig:MSEs} shows mean squared errors (MSEs) for the risk estimates. The figure shows the integrated mean sequared error as well as the mean squared error in the optimal $\lambda$ (the $\lambda$ that minimizes risk as estimated from the Monte Carlo estimate of the risk based on 1000 replications). The cross-validation risk estimates generally have the largest MSEs, while $\widehat{\mathrm{Risk}}_{\textrm{df}}$ has considerably smaller MSEs. This is true even when the variance is estimated except for $n = 50$ and $p = 2000, 20000$. From this figure we see that $\widehat{\mathrm{Risk}}_{\textrm{df}}$ does have a larger MSE than $\widehat{\mathrm{Risk}}_{\mathrm{lasso}}$. Moreover, for $n / p$ large the estimation of $\sigma$ does not affect the MSE of the risk estimates much. 

For this simulation study we also recorded the number of selected predictors as well as the computational time for evaluating and tuning the different estimators. The results can be found as Figure 1 in the supplementary material. The lasso-OLS estimator selects fewer predictors than lasso, but when the variance is estimated, the number of selected predictors is increased -- this is particularly so when $n / p$ is small. The lasso estimator using \eqref{eq:risk_est_lasso} for tuning is fastest, which is unsurprising as the computation of the lasso path is part of all estimators. Moreover, the lasso-OLS estimator using \eqref{eq:riskhat} for tuning is about a factor 4 faster than using 5-fold cross-validation for tuning and about a factor 8 faster than 10-fold cross-validation. Thus the added computation of the smoothed term to the estimate of degrees of freedom in  \eqref{eq:riskhat} has an insignificant effect on the computation time. 

\clearpage 

\subsection{Results from study II}

\begin{figure}
    	\centering
    	\includegraphics[width=\textwidth]{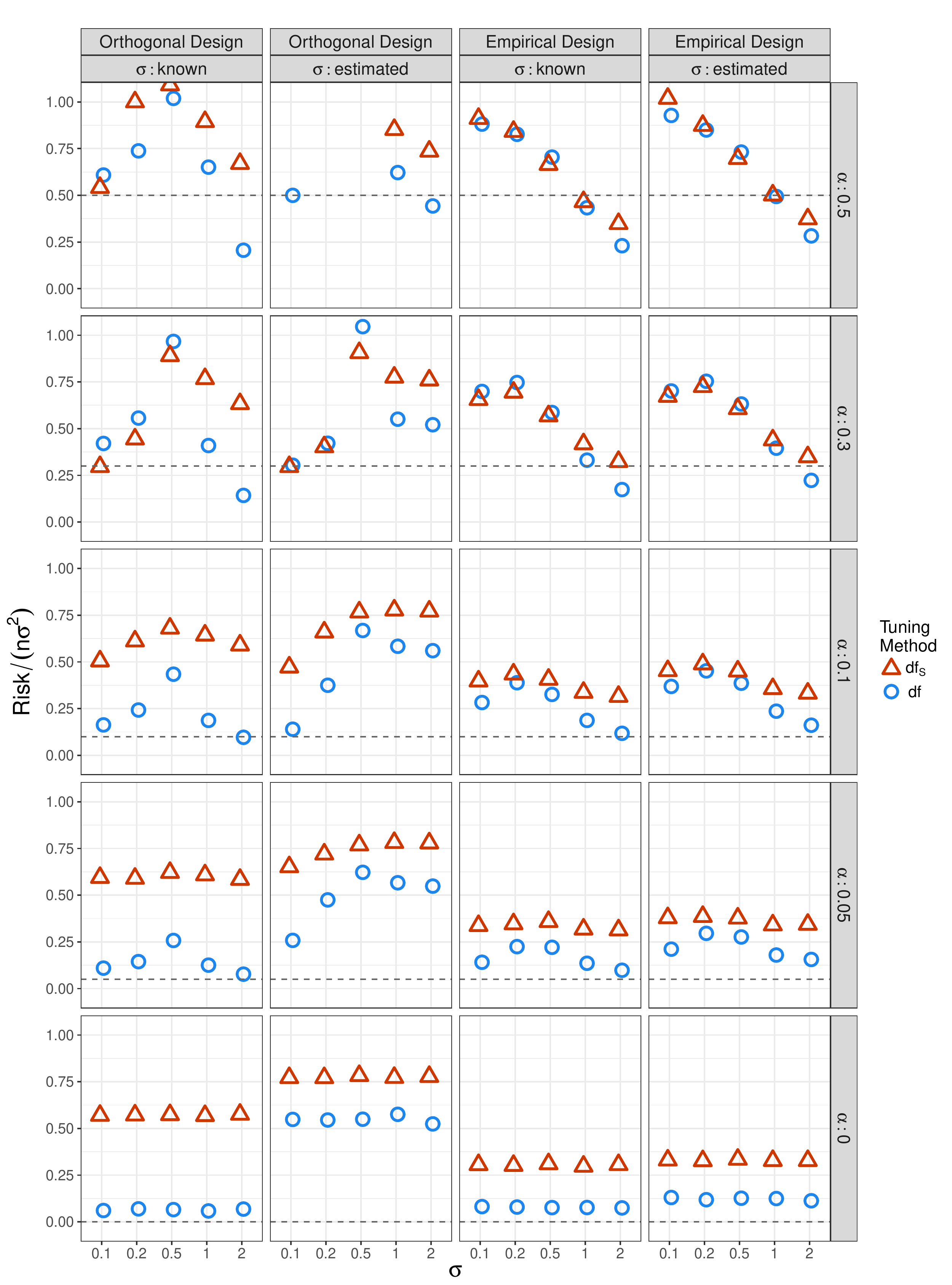}
    	\caption{Risk relative to $\sigma^2 n$ for the estimators $\hat{\mu}_{\textnormal{l-OLS}}^{\hat{\lambda}_{\mathrm{df}_{\mathrm{S}}}}$ and $\hat{\mu}_{\textnormal{l-OLS}}^{\hat{\lambda}_{\mathrm{df}}}$ for orthogonal and empirical designs with $n=100$ and $\gamma=1$. The dashed line is $\lceil n \alpha \rceil / n \simeq \alpha$, the relative risk for the oracle-OLS estimator.}
    	\label{fig:box}
\end{figure}	

\begin{figure}
    	\centering
    	\includegraphics[width=\textwidth]{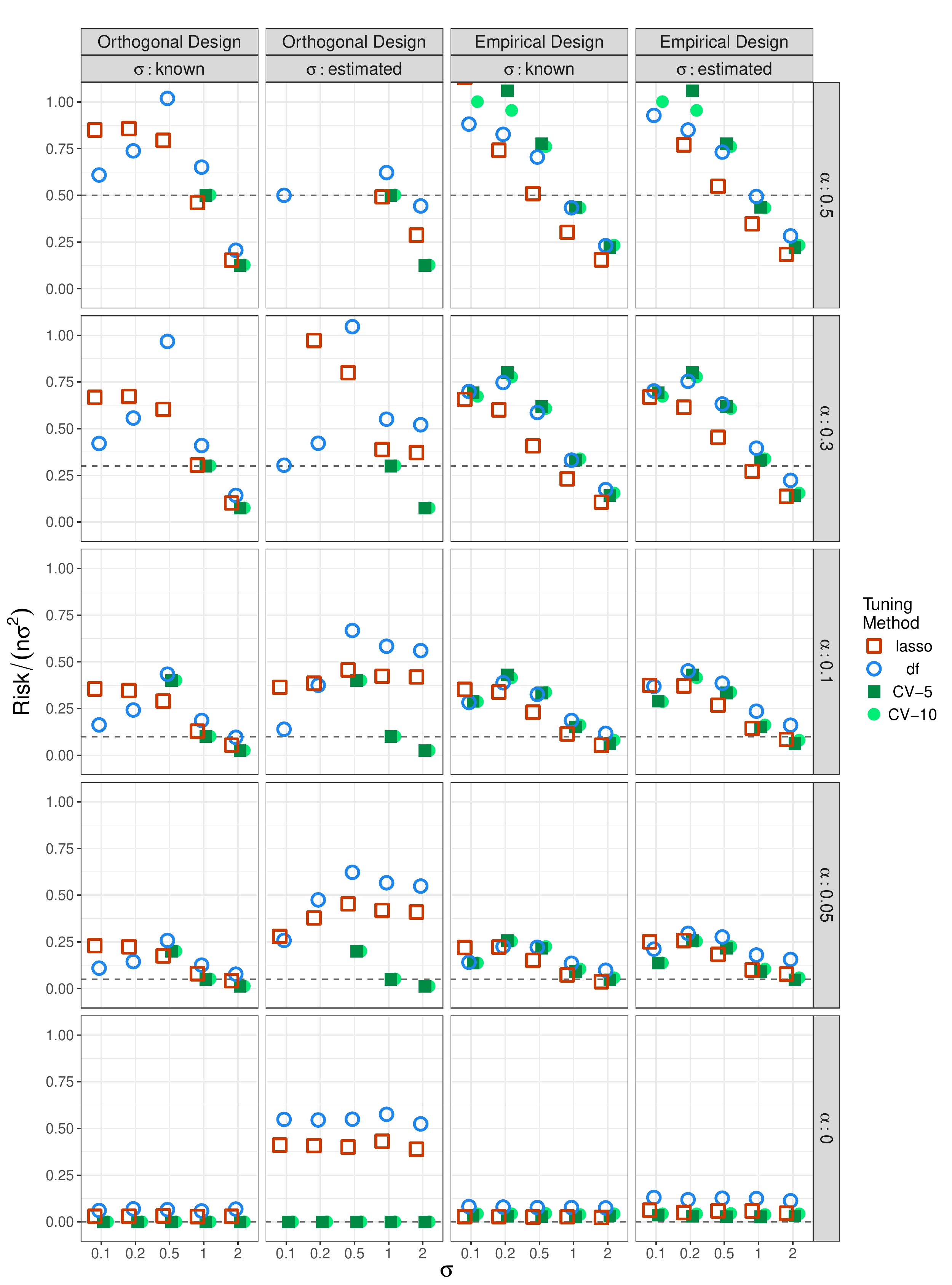}
    	\caption{Risk relative to $\sigma^2 n$ for the estimators $\hat{\mu}_{\textnormal{l-OLS}}^{\hat{\lambda}_{\mathrm{df}}}$,
    	$\hat{\mu}_{\textnormal{l-OLS}}^{\hat{\lambda}_{\textnormal{CV-5}}}$,     	    	$\hat{\mu}_{\textnormal{l-OLS}}^{\hat{\lambda}_{\textnormal{CV-10}}}$ and 
    	$\hat{\mu}_{\mathrm{lasso}}^{\hat{\lambda}_{\mathrm{lasso}}}$ for orthogonal and empirical designs with $n=100$ and $\gamma=1$. The dashed line is $\lceil n \alpha \rceil / n \simeq \alpha$, the relative risk for the oracle-OLS estimator.}
    	\label{fig:box2}
\end{figure}

Firstly, we discuss the comparison of the two tuning methods $\mathrm{df}$ and $\mathrm{df}_{\mathrm{S}}$ for the lasso-OLS estimator. The purpose of this comparison is to highlight the effect of correctly adjusting for the variable selection in the estimation of degrees of freedom via the term $\widehat{\partial}$. Secondly, we discuss the comparison of $\mathrm{df}$ to $\textnormal{CV-5}$, $\textnormal{CV-10}$ and lasso. The purpose of this second comparison is two-fold. It provides a comparison of our proposed tuning method, $\mathrm{df}$, to cross-validation based tuning, and it provides a comparison of lasso-OLS to lasso in terms of predictive performance. 

Figure \ref{fig:box} shows the results for the two tuning methods $\mathrm{df}$ and $\mathrm{df}_{\mathrm{S}}$ in the orthogonal and empirical designs with $\gamma = 1$ and $n = 100$. The supplementary material contains the results for all the other design parameters. Tuning $\lambda$ by using $\mathrm{dim}(\hat{S}^{\lambda}) + \widehat{\partial}$ as an estimate of degrees of freedom is generally superior to using $\mathrm{dim}(\hat{S}^{\lambda})$ and in the worst cases at least comparable. The differences are largest for the lowest signal-to-noise ratios. The benefit of using $\mathrm{dim}(\hat{S}^{\lambda}) + \widehat{\partial}$ generally increases with the dimension $n$, and it increases with decreasing signal-to-noise ratio. Furthermore, when the number of non-zero parameters is large and the signal-to-noise ratio is low (specifically, $\gamma = 0.9$, $\alpha$ large and $\sigma$ large), $\hat{\mu}_{\textnormal{l-OLS}}^{\hat{\lambda}_{\mathrm{df}}}$ clearly outperforms the oracle-OLS estimator, while $\hat{\mu}_{\textnormal{l-OLS}}^{\hat{\lambda}_{\mathrm{df}_{\mathrm{S}}}}$ is comparable or worse than the oracle-OLS estimator. Neither of the estimators performs well for small variances and large signal-to-noise ratios. For the orthogonal design the estimation of the variance incurs a clear performance loss, which is not the case for the other designs. We ascribe this to the variance estimator being particularly poor for the orthogonal design.

Figure \ref{fig:box2} shows the results for $\mathrm{df}$, $\textnormal{CV-5}$, $\textnormal{CV-10}$ and lasso for the orthogonal and empirical designs with $\gamma = 1$ and $n = 100$. The results for the remaining design parameters are found in the supplementary material. For the orthogonal design cross-validation is not an appropriate tuning method, since $\widehat{\mathrm{Risk}}_{\textnormal{CV-K}}$ is constant in $\lambda$. This relates to the fact that the folds cannot be considered replications of the same distribution. Consequently, for the orthogonal design, the tuning methods based on degrees of freedom have clear advantages. On the other hand, the estimation of $\sigma$ has a quite large negative effect for precisely the orthogonal design. 

When restricting attention to the non-orthogonal designs we observe that the tuning methods are quite comparable (see the supplementary material). None of the tuning methods are generally superior or inferior to the others and their performance depends on both design type, signal-to-noise ratio and the signal decay parameter $\gamma$. The lasso estimator deviates most from the others, which is mainly due to this being a different estimator. It performs best at low signal-to-noise ratios, while lasso-OLS using either cross-validation of $\mathrm{df}$ tuning  performs better at high signal-to-noise ratios ($\alpha$ large, $\sigma$ small and $\gamma = 1$). Cross-validation appears to perform best for highly correlated designs ($\rho$ large).  

The results for the non-Gaussian error distributions are included in the supplementary material as well. There are no major differences when compared to the Gaussian error distribution, with the most notable change being that lasso losses some of its performance for the $t$-distributed noise. The tuning based on $\mathrm{df}$ seems to be less affected. Still, all the tuning methods are generally comparable except for orthogonal designs. Since cross-validation does not rely on a Gaussian noise assumption, these results suggest that our proposed tuning method based on $\mathrm{df}$ is appropriate even in non-Gaussian settings.

	\section{Best Subset Selection}
        \label{sec:best}

        Example \ref{exam:U_A} demonstrates that \eqref{eq:set_cond} holds for other estimators than lasso-OLS, and Theorem \ref{theo:scaling} holds, in particular, for best subset selection in the Lagrangian formulation, which corresponds to $\mathrm{Pen}(\cdot) = \|\cdot\|_0$ in Example \ref{exam:U_A}. Theorem \ref{theo:debiasedlasso} does, however, only partly extend to best subset selection. In this section we demonstrate that this may still provide a practically useful estimate of degrees of freedom. 

The best subset selection estimator of $\mu$ with tuning parameter $\lambda>0$, denoted by $\hat{\mu}_{\mathrm{bs}}^{\lambda}$, is
		\[ \hat{\mu}_{\mathrm{bs}}^{\lambda}=X\hat{\beta}^{\lambda} \quad\mathrm{ where }\quad \hat{\beta}^{\lambda}=\underset{\beta}{\argmin} \ \frac{1}{2}\|Y-X\beta\|^2_2+\lambda\|\beta\|_0.   \]
		It can be written on the form $\hat{\mu}_{\mathrm{bs}}^{\lambda}=\sum_{A\in\{1,...,p\}}{1_{U_A^{\lambda}}\Pi_A}$ (Lebesgue a.e.), where
		\begin{equation}
		\label{eq:U_A_bs}
		U_A^{\lambda}:=\left\{ y\in \mathbb{R}^n \ \bigg\vert \ \lambda|A|-\frac{1}{2}\|\Pi_Ay\|^2_2 < \underset{B\in \{1,...,p\}\setminus A}{\min}\lambda|B|-\frac{1}{2}\|\Pi_By\|^2_2 \right\}, \qquad A\subset\{1,...,p\}. 
		\end{equation}
                It is straightforward to verify that $\hat{\mu}_{\mathrm{bs}}^{\lambda}$ fulfils Assumption \ref{ass:assumption1} except \ref{ass:assumption1}\eqref{itm:ass_3}, which follows by 
		Lemma \ref{lem:semialgebraic} in the appendix. Hence Theorem \ref{theo:singular_repres} applies to $\hat{\mu}_{\mathrm{bs}}^{\lambda}$. 	 
	
 \begin{figure}
		\centering
		\includegraphics[width=0.8\textwidth]{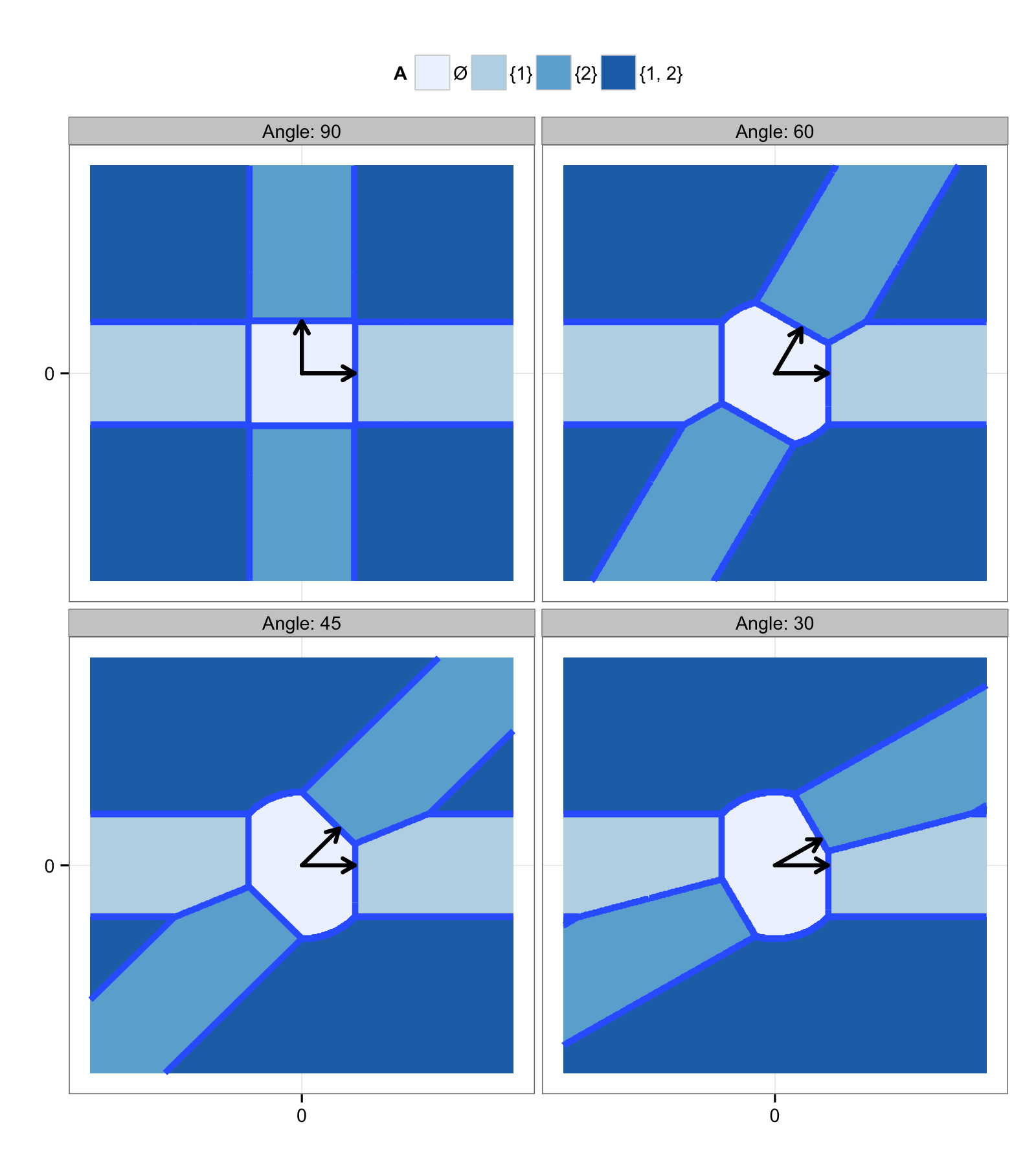}
		\caption{Illustrations of the decomposition of $\mathbb{R}^2$ into the four sets $U_{\emptyset}^1$, $U_{\{1\}}^1$, $U_{\{2\}}^1$ and $U_{\{1, 2\}}^1$ according to the best subset selection estimator in the Lagrangian formulation with $\lambda = 1$. The set $U_{\emptyset}^1$ consists of the points projected onto the 0-dimensional space $\{0\}$, the sets $U_{\{1\}}^1$, $U_{\{2\}}^1$ to the projections onto one of the two $1$-dimensional subspaces and $U_{\{1, 2\}}^1$ to the identity map. The decomposition depends on the angle between the two columns in $X$.}
		\label{fig:bss_decomp}
	\end{figure}

	 From \eqref{eq:U_A_bs} we note that the outer unit normal to $\partial U_{A_1}^{\lambda}$ on $\overline{U}_{A_1}^{\lambda} \cap\overline{U}_{A_2}^{\lambda}$ equals $(\Pi_{A_2}-\Pi_{A_1})y$ normalized to have norm 1. Theorem \ref{theo:singular_repres} yields
	\begin{equation*}
	\begin{aligned}
	\mathrm{df}(\hat{\mu}_{\mathrm{bs}}^{\lambda})-\mathrm{df}_{\mathrm{S}}(\hat{\mu}_{\mathrm{bs}}^{\lambda}) &= \frac{1}{2}\sum_{A_1\neq A_2}{  \int_{ \overline{U}_{A_1}^{\lambda}\cap \overline{U}_{A_2}^{\lambda}}{ \frac{\left\langle (\Pi_{A_2}-\Pi_{A_1})y,(\Pi_{A_2}-\Pi_{A_1})y\right\rangle}{\|(\Pi_{A_2}-\Pi_{A_1})y\|_2}  \psi(y)\  d\mathcal{H}^{n-1}(y) } }\\
	&=\frac{1}{2}\sum_{A_1\neq A_2}{  \int_{ \overline{U}_{A_1}^{\lambda}\cap \overline{U}_{A_2}^{\lambda}}{   \left\|(\Pi_{A_2}-\Pi_{A_1})y\right\|_2 \psi(y)\ d\mathcal{H}^{n-1}(y) } },
	\end{aligned}
	\end{equation*}
	which proves that $\mathrm{df}>\mathrm{df}_{\mathrm{S}}$ for best subsection selection. Moreover, Proposition \ref{theo:scaling} and Example \ref{exam:U_A} yields
	\begin{equation*}
{\small
	\begin{aligned}
	-2\lambda\partial_{\lambda}\mathrm{df}_{\mathrm{S}}(\hat{\mu}_{\mathrm{bs}}^{\lambda}) &=\frac{1}{2}\sum_{A_1\neq A_2}{  \int_{ \overline{U}_{A_1}^{\lambda}\cap \overline{U}_{A_2}^{\lambda}}{  \psi(y) \frac{ \langle y, (\Pi_{A_2}-\Pi_{A_1})y \rangle}{\left\|(\Pi_{A_2}-\Pi_{A_1})y\right\|_2} (|A_2|-|A_1|) \ d\mathcal{H}^{n-1}(y) } }.
	\end{aligned}
}
	\end{equation*}
	For $\mathrm{col}(X_{A_1}) \subseteq \mathrm{col}(X_{A_2})$ and $\mathrm{rank}(X_{A_2})=\mathrm{rank}(X_{A_1})+1$, we see that the integrands in the two identities above coincide. Hence, if we define
	\begin{align*}
	\mathcal{A}_1&:=\left\{ A_1,A_2\subseteq\{1,...,p\} \ \bigg\vert \ \begin{aligned}
	\mathrm{col}(X_{A_1}) \subseteq \mathrm{col}(X_{A_2}) \text{ and } \\
	\mathrm{rank}(X_{A_2})=\mathrm{rank}(X_{A_1})+1
	\end{aligned}   \right\} \quad \text{and}\\
		\mathcal{A}_2&:=\left\{ A_1,A_2\subseteq\{1,...,p\} \ \bigg\vert \mathrm{col}(X_{A_1}) \neq \mathrm{col}(X_{A_2}) \text{ and } \begin{aligned}
		(A_1,A_2)\notin \mathcal{A}_1\\
		(A_2,A_1)\notin \mathcal{A}_1
		\end{aligned}\right\},
	\end{align*}
	then
$$\mathrm{df}(\hat{\mu}_{\mathrm{bs}}^{\lambda})-\mathrm{df}_{\mathrm{S}}(\hat{\mu}_{\mathrm{bs}}^{\lambda}) = -2\lambda\partial_{\lambda}\mathrm{df}_{\mathrm{S}}(\hat{\mu}_{\mathrm{bs}}^{\lambda}) + R$$
where
{\footnotesize 
$$R = \frac{1}{2}\sum_{(A_1,A_2)\in\mathcal{A}_2}{ \int_{ \overline{U}_{A_1}^{\lambda}\cap \overline{U}_{A_2}^{\lambda}} {  \frac{   \langle (\Pi_{A_2}-\Pi_{A_1})y, \left(\Pi_{A_2}-\Pi_{A_1} - (|A_2|-|A_1|)I_n\right)y \rangle  }{\left\|(\Pi_{A_2}-\Pi_{A_1})y\right\|_2}\psi(y) \ d\mathcal{H}^{n-1}(y) }}$$
} 

	\begin{figure}
		\centering
		\includegraphics[width = \textwidth]{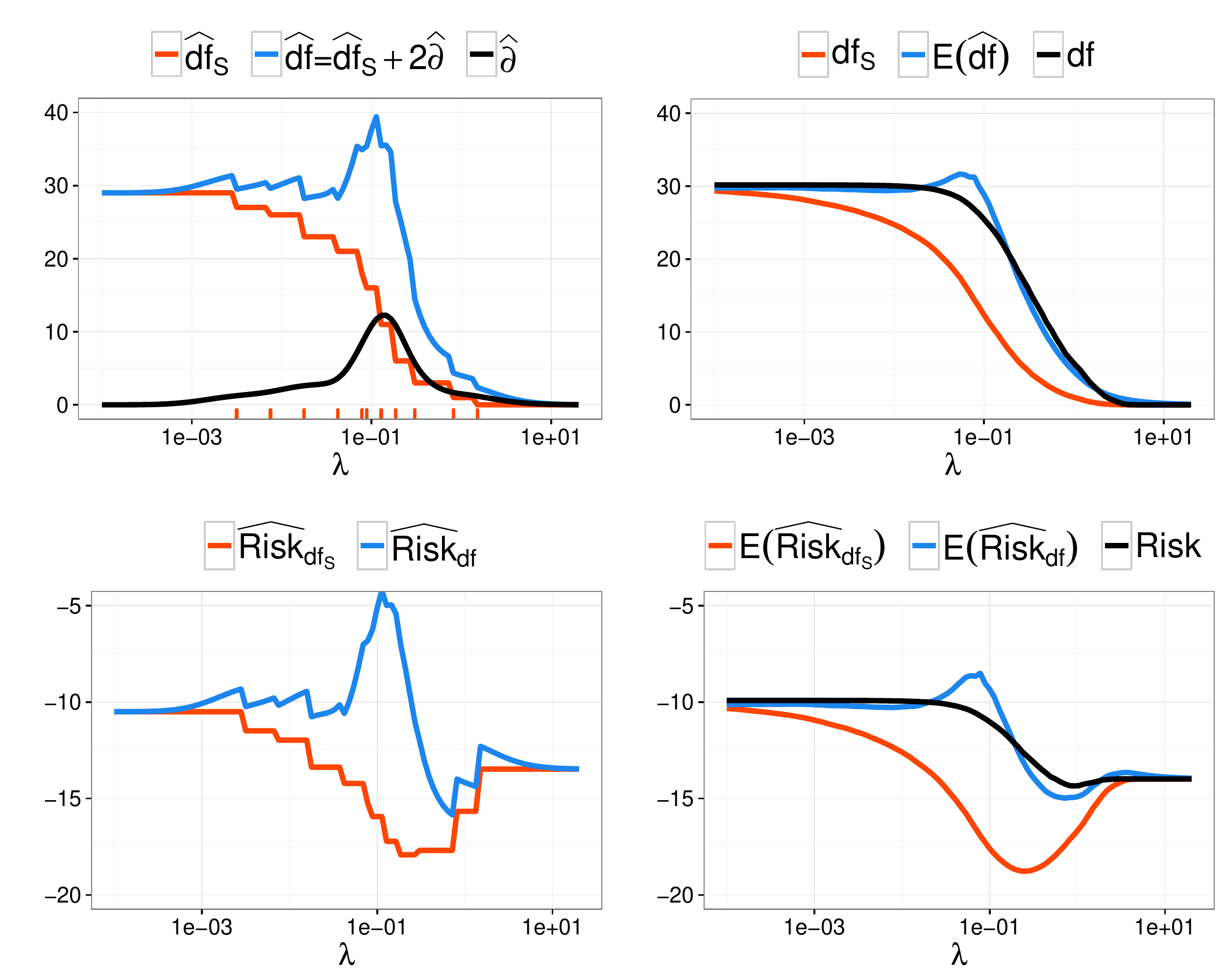}
		\caption{
Left: Realization of the estimates of degrees of freedom $\hat{\mathrm{df}}_S = \mathrm{dim}(\hat{S}^{\lambda})$ and $\hat{\mathrm{df}} = \mathrm{dim}(\hat{S}^{\lambda}) + 2\widehat{\partial}$ as well as the correction term $\widehat{\partial}$ as a function of $\log(\lambda)$ for best subset selection (top) and corresponding estimates of the risk (bottom). Right: Similar to the left but mean values of the estimates obtained by averaging over 1000 samples along with the degrees of freedom $\mathrm{df} = \mathrm{df}(\hat{\mu}_{\mathrm{bs}}^{\lambda})$ obtained from the 1000 samples using the covariance definition \eqref{eq:df_def}. The design parameters were: $\sigma=0.5$, $n=p=30$, $\gamma =1$, $\alpha =0.1$ and the design type was (S) with constant correlation of $\rho = 0.1$ (see Section \ref{sec:sim}).}
		\label{fig:bss}
	\end{figure}	

The usefulness of this hinges on $R$ being small. For $X$ orthogonal we have already demonstrated that $R = 0$ as $\hat{\mu}_{\mathrm{bs}}^{\lambda}$ then coincides with lasso-OLS, and in this case $\overline{U}_{A_1}^{\lambda}\cap \overline{U}_{A_2}^{\lambda}$ has Hausdorff measure zero for all $(A_1, A_2) \in \mathcal{A}_2$. For non-orthogonal $X$ this is no longer true, see Figure \ref{fig:bss_decomp}. For best subset selection there will generally be boundaries of non-zero Hausdorff measure between many more of the sets $\overline{U}_{A}^{\lambda}$ -- boundaries that correspond to including or excluding more than one predictor at the time or replacing predictors. Compare this with lasso-OLS and Figure \ref{fig:dlasso_decomp}. However, by continuity in $X$ we have $R \to 0$ for $X$ tending to an orthogonal matrix, and we can expect $R$ to be small for matrices that are not too far from orthogonal matrices. Thus we expect 
\begin{equation}
\label{eq:approxdf}
\mathrm{df}_{\mathrm{S}}(\hat{\mu}_{\mathrm{bs}}^{\lambda})  -2\lambda\partial_{\lambda}\mathrm{df}_{\mathrm{S}}(\hat{\mu}_{\mathrm{bs}}^{\lambda})
\end{equation}
	to be a useful approximation for $\mathrm{df}(\hat{\mu}_{\mathrm{bs}}^{\lambda})$ also for non-orthogonal $X$. 

Using the same procedure for estimating the correction $-2\lambda\partial_{\lambda}\mathrm{df}_{\mathrm{S}}(\hat{\mu}_{\mathrm{bs}}^{\lambda})$ as outlined in Section \ref{sec:l-OLS} -- using $2\widehat{\partial}$ instead of $\widehat{\partial}$ -- we used simulations to investigate if \eqref{eq:approxdf} was actually a good approximation of $\mathrm{df}(\hat{\mu}_{\mathrm{bs}}^{\lambda})$. Figure \ref{fig:bss} shows the results using the same configurations as in Figure \ref{fig:dlasso}, except that $n$ was lowered to 30 due to computational constraints. The conclusion from this and other similar simulations (not shown) is that even with non-orthogonal designs, \eqref{eq:approxdf} is a practically useful approximation. That is,  $-2\lambda\partial_{\lambda}\mathrm{df}_{\mathrm{S}}(\hat{\mu}_{\mathrm{bs}}^{\lambda})$ accounts for the majority of the increase in the degrees of freedom due to variable selection. 

	\clearpage
	
	\section{Discussion}
 	We have provided a new representation of degrees of freedom for a broad class of discontinuous, piecewise Lipschitz estimators. This representation provides us with a deeper insight into the effect of variable selection, among other things, on the effective dimension of the statistical model and the estimator used. We have demonstrated that for lasso-OLS it was, moreover, possible to derive a practically useful estimator of the degrees of freedom based on the general representation, and we also suggest that a similar estimator can be useful for best subset selection. The estimator was based on relating the derivative of $\lambda \mapsto \mathrm{df}_{\mathrm{S}}(\hat{\mu}^{\lambda})$ to the discontinuities of the estimator $\hat{\mu}^{\lambda}$ as expressed via the integral representation of $\mathrm{df}(\hat{\mu}^{\lambda}) - \mathrm{df}_{\mathrm{S}}(\hat{\mu}^{\lambda})$. This does, indeed, make some intuitive sense as the first expresses the mean jump of degrees of freedom per unit change of $\lambda$ and the other (in some sense) the mean discontinuity of degrees of freedom per unit change of $y$. Changing $\lambda$ for fixed $y$ or changing $y$ for fixed $\lambda$ are dual operations, and it is not surprising that we can relate the numbers. 

A simulation study demonstrated that the risk of the lasso-OLS estimator can be estimated effectively by using our proposed estimate of degrees of freedom. Our proposal did not incur any substantial computational penalty, nor did it incur a substantial increase in the variance of the risk estimate. The simulation study also showed that lasso-OLS can be effectively tuned by minimising our proposed risk estimate, and that the resulting computations are faster than using cross-validation. The resulting lasso-OLS estimator selects fewer predictors than lasso with a comparable predictive performance, but it is computationally more expensive. 

If we were to generalize our results to other estimators that include a tuning parameter, we expect that it is only the derivative of the part of $\mathrm{df}_{\mathrm{S}}(\hat{\mu}^{\lambda})$ that corresponds to jumps that can be related to $\mathrm{df}(\hat{\mu}^{\lambda}) - \mathrm{df}_{\mathrm{S}}(\hat{\mu}^{\lambda})$. That is, in general, $\lambda \mapsto \mathrm{div}(\hat{\mu}^{\lambda})$ will have jumps as well as smooth but non-constant pieces, and it is only the expectation of the jump part that we expect can be related to $\mathrm{df}(\hat{\mu}^{\lambda}) - \mathrm{df}_{\mathrm{S}}(\hat{\mu}^{\lambda})$. We believe that our suggested estimator of degrees of freedom may actually be generalizable to a number of discontinuous estimators involving variable selection as well as shrinkage. The requirement will be that the estimator has one or more tuning parameters and that it is computed on a grid or along a path of these. Then we can potentially estimate the derivative of the divergence of the estimator as a function of the tuning parameter(s). It is an ongoing research project to investigate this in detail. 

For best subset selection we did not provide any bounds on the residual $R$ in the approximation of $\mathrm{df}(\hat{\mu}^{\lambda}) - \mathrm{df}_{\mathrm{S}}(\hat{\mu}^{\lambda})$. It would, indeed, be very interesting to investigate this approximation in more detail. It would, in particular, be interesting to understand if it in any way can be seen as a ``first order approximation'' and whether there are higher order terms worth including in some cases.  

Finally, we have restricted attention to Gaussian noise in the theoretical derivations. Like Stein's classical lemma, Theorem \ref{theo:singular_repres} crucially relies on this assumption. Our simulation study demonstrated some robustness towards deviations from this assumption. However, extensions of Stein's lemma to non-Gaussian distributions do exist (see, e.g., \cite{dalalyan2008}), but further investigations are required to determine if similar extensions can be made in the more general framework presented in this paper. 

	
	\appendix
	\section{Additional results and Proofs}

\subsection{Semialgebraic sets}
	\label{sec:semialg}

        Observe that for $A$ and $B$ subsets of $\mathbb{R}^n$ it holds that		
		\begin{equation}
		\label{eq:boundary}
			\begin{aligned}			
			\partial A&=\partial (A^c),\\
			\partial (A\cup B)&\subseteq\partial A\cup \partial B,\\
			\partial (A\cap B)&\subseteq\partial A\cup \partial B.
			\end{aligned}
		\end{equation}
		Especially, the family of sets
		\begin{equation}
		\label{eq:good_sets}
		\left\{ E\in \mathcal{B}(\mathbb{R}^n) \ \Bigg\vert  \begin{aligned}
		r\mapsto  \mathcal{H}^{n-1}\left(\partial E \cap B(0,r) \right)\\
		\text{  is polynomially bounded }
		\end{aligned} \right\}
		\end{equation}
		is stable under complement, finite union and finite intersection. This is a useful observation when we want to verify Assumption \ref{ass:assumption1}\eqref{itm:ass_3}. 
	
	The following Lemma shows that \textit{semialgebraic sets} belong to the family given by \eqref{eq:good_sets}. A semialgebraic set is finite union of finite intersections of sets of the form $(P=0)$ and $(Q > 0)$, where $P$ and $Q$ are polynomials. A multivariate polynomial is of the form (using multi-index notation)
	\[ P(x) = \sum_{\alpha\in A}{a_{\alpha}x^{\alpha}}, \quad a_{\alpha}\in \mathbb{R} \text{ for each }\alpha\in A,   \]
	with $A \subseteq \mathbb{N}^n$ finite.

	\begin{lem}
		\label{lem:semialgebraic}
		%
		If $E$ is semialgebraic then $r\mapsto  \mathcal{H}^{n-1}\left(\partial E \cap B(0,r) \right)$ is polynomially bounded.
	\end{lem}
	\begin{proof}
		By the stability under finite set operations of the family given by  \eqref{eq:good_sets} it suffices to show that $r\mapsto  \mathcal{H}^{n-1}\left((P=0) \cap B(0,r) \right)$ is polynomially bounded for any nonzero polynomial $P$. But this follows from Corollary 1 in \cite{loi2014}, which implies that 
		\[ \mathcal{H}^{n-1}((P=0) \cap B(0,r)) \leq \frac{\mathrm{deg}(P)\pi^{\frac{n+1}{2}}}{\Gamma\left(\frac{n}{2}\right)}r^{n-1} \]	
		for any nonzero polynomial $P$ with $\mathrm{deg}(P)=\underset{a_{\alpha}\neq 0}{\max}\ |\alpha|$ denoting the degree of $P$.
	\end{proof}

\subsection{Proof of Theorem \ref{theo:singular_repres}}
	\label{sec:proofs}

	The following Lemma characterizes the outer unit normal vectors $\eta_i$ for $i = 1, \ldots, N$. 
	\begin{lem}
		\label{lem:outer_unit_normal}
		Under Assumption \ref{ass:assumption1} the following holds:
		\begin{enumerate}[(a)]
			\item \label{itm:normal_1} $\eta_i=0$ $\mathcal{H}^{n-1}$ a.e. on $\partial U_i \setminus \bigcup_{j\neq i}\overline{U}_j$ for each  $i=1,...,N$.
			\item \label{itm:normal_2} $\eta_i = - \eta_j$ $\mathcal{H}^{n-1}$ a.e. on $\partial U_i \cap \partial U_j$ with $i\neq j$.
			\item \label{itm:normal_3} $\eta_i = 0$ $\mathcal{H}^{n-1}$ a.e. on $\partial U_i \cap \partial U_j \cap \partial U_k$ with $i, j, k$ distinct.
		\end{enumerate}
	\end{lem}
	\begin{proof}
%
%
		Firstly, note that the unit outer normal $\eta_i$ on $\partial U_i$ vanishes outside \textit{the measure theoretic boundary} $\partial_*U_i$, see Definition 5.8 in \cite{evans1991measure}. 
		Moreover, these two types of boundaries relates to \textit{the reduced boundary} $\partial^*U_i$ (see Definition 5.7 in \cite{evans1991measure}) by the inclusions:
		\[ \partial^*U_i\subseteq \partial_*U_i \subseteq \partial U_i. \]
		Furthermore, $\mathcal{H}^{n-1}(\partial_*U_i\setminus\partial^*U_i)=0$ (see Lemma 5.8.1 in \cite{evans1991measure}). All in all, we see that the Lemma holds if we can show the following claims: 
		\begin{equation}
		\label{eq:claim2}
			\begin{aligned}
				\partial^*U_i \subseteq& \bigcup_{l\neq i} \overline{U}_l \\
				\eta_i=-\eta_j \text{ on }& \partial^*U_i \cap \partial^*U_j\\
				\partial^*U_i \cap \partial^*U_j& \cap \partial^*U_k = \emptyset
			\end{aligned}
		\end{equation}
		holds for all $i, j, k$ distinct.
		
		To prove the claims, define for each $i$ and $r>0$ the sets
		\begin{align*}
			U_i^r(x)&=\{ y \mid r(y-x) + x \in U_i \},\\
			H_i(x) 	&=\{ y \mid \langle \eta_i , y - x\rangle \leq 0 \}.
		\end{align*}  
		Note that $\{U_i^r(x)\}_i$ are still disjoint. By Theorem 5.7.1 in \cite{evans1991measure}
		\[ 1_{U_i^r(x)}\xrightarrow{r\rightarrow 0} 1_{H_i(x)} \text{ in } L^1_{\mathrm{loc}}(\mathbb{R}^n) \text{ for all } x\in \partial^* U_i. \]
		
		Therefore, if there existed $x\in \partial^*U_i \cap \partial^*U_j \cap \partial^*U_k$ for $i,j,k$ distinct, then 
		\begin{equation}
		\label{eq:blow_up_boundary}
		1_{U_i^r(x)\cup U_j^r(x) \cup U_k^r(x)} \xrightarrow{r\rightarrow 0} 1_{H_i(x)}+1_{H_j(x)}+1_{H_k(x)} \text{ in } L^1_{\mathrm{loc}}(\mathbb{R}^n),
		\end{equation}
		which is impossible as the right hand side is not Lebesgue a.e. an indicator. By the same argument one can deduce that $\eta_i=-\eta_j$ must hold for $x\in \partial^*U_i \cap \partial^*U_j$ and that any $x\in \partial^*U_i$ cannot belong to the open set $(\bigcup_{l\neq i}\overline{U}_l)^c$. 		
	\end{proof}

	\begin{proof}[Proof of Theorem \ref{theo:singular_repres}]
		For $i = 1,...,N$ Gauss-Green's formula (see Theorem 5.8.1 in \cite{evans1991measure} and Theorem 4.5.6 in \cite{federer1969geometric}) gives that
		\begin{equation}
		\label{eq:GaussGreen}
		\int_{U_i}{\mathrm{div} (f) \ dm}=\int_{\partial U_i}{\langle f,\eta_i\rangle \ d\mathcal{H}^{n-1}}
		\end{equation}
		for all Lipschitz continuous vector fields $f$ with compact support. Here $\eta_i$ denotes the outer unit normal of $\partial U_i$, which is well defined and nonzero on a subset of $\partial U_i$ and zero everywhere else by definition. 
		
		Let $(g_r)_r$ be a sequence of smooth functions with
		\[ g_r(x)=\begin{cases}
		1 & \text{if } x \in B(0,r)\\
		0 & \text{if } x \notin B(0,r+1)
		\end{cases} \]
		and $(g_r)_r$ and $(Dg_r)_r$ uniformly bounded. Since $\hat{\mu}_i$ is Lipschitz continuous on $\overline{U}_i\cap B(0,r+1)$ Kirzbraun's theorem ensures that $\hat{\mu}_i$ has a Lipschitz extension, $\hat{\mu}_i^r:\mathbb{R}^n\rightarrow\mathbb{R}^n$. Then $f_r=g_r\psi\hat{\mu}_i^r$ is Lipschitz continuous with compact support and $g_r \hat{\mu}^r_i = g_r \hat{\mu}$ on $U_i$. Then \eqref{eq:GaussGreen} applied to $f_r$ yields
		\[ \int_{\partial  U_i}{g_r\psi \langle\hat{\mu}_i, \eta_i\rangle \ d\mathcal{H}^{n-1}} = \int_{U_i}{g_r\psi \mathrm{div} (\hat{\mu}_i)  \ dm}+ \int_{U_i}{\left\langle g_rD\psi + \psi Dg_r, \hat{\mu}_i \right\rangle  dm}. \]
		Due to Assumption \ref{ass:assumption1} all integrands above are dominated by integrable functions, and by letting $r \to \infty$ Lebesgue's Dominated Convergence Theorem yields
		\[ \int_{\partial U_i}{\psi \langle\hat{\mu}_i, \eta_i\rangle \ d\mathcal{H}^{n-1}} = \int_{U_i}{\psi \mathrm{div} (\hat{\mu}_i) \ dm}+ \int_{U_i}{\langle D\psi , \hat{\mu}_i\rangle \ dm}. \]
		By summing over $i$ we get
		\begin{equation}
		\label{eq:GaussGreen_intermediate} 
		\mathrm{df}(\hat{\mu})=\mathrm{df}_{\mathrm{S}}(\hat{\mu}) - \sum_{i}{  \int_{ \partial U_i}{\psi \langle\hat{\mu}_i, \eta_i\rangle d\mathcal{H}^{n-1} } }. 
		\end{equation}
		By Lemma \ref{lem:outer_unit_normal} we see that
		\begin{align*}
		\label{eq:GaussGreen_intermediate2} 
		\mathrm{df}(\hat{\mu})&=\mathrm{df}_{\mathrm{S}}(\hat{\mu}) - \sum_{j\neq i}{  \int_{ \partial U_i \cap \partial U_j}{\psi \langle\hat{\mu}_i, \eta_i\rangle d\mathcal{H}^{n-1} } } \\
		&= \mathrm{df}_{\mathrm{S}}(\hat{\mu}) + \frac{1}{2} \sum_{j\neq i}{  \int_{ \partial U_i \cap \partial U_j}{  \langle\hat{\mu}_j - \hat{\mu}_i,\eta_i\rangle\psi  d\mathcal{H}^{n-1} } }. 
		\end{align*}
		Since $\eta_i$ vanishes on $\partial  U_i \cap \partial  U_j \setminus (\overline{U}_i \cap \overline{U}_j)$ for $i\neq j$ we have proven \eqref{eq:df_struct}. 
	\end{proof}

	\bibliography{bibliography}{}
	\bibliographystyle{agsm}

\end{document}


\onehalfspacing
	\pagenumbering{gobble}
	
\title[Suppl. Mat. DF for piecewise Lipschitz estimators]{Supplementary Material: Degrees of freedom for piecewise Lipschitz estimators}
\author[F. R. Mikkelsen]{Frederik Riis Mikkelsen}
\address{Department of Mathematical Sciences,
University of Copenhagen,
Universitetsparken 5,
2100 Copenhagen \O,
Denmark}

\email[Corresponding author]{frm@math.ku.dk}
\author[N. R. Hansen]{Niels Richard Hansen} 
\email{Niels.R.Hansen@math.ku.dk}

\maketitle

\noindent

\section{Computational Time and Number of Selected Predictors}

\begin{figure}[h]
    	\centering
    	\includegraphics[width=.9\textwidth]{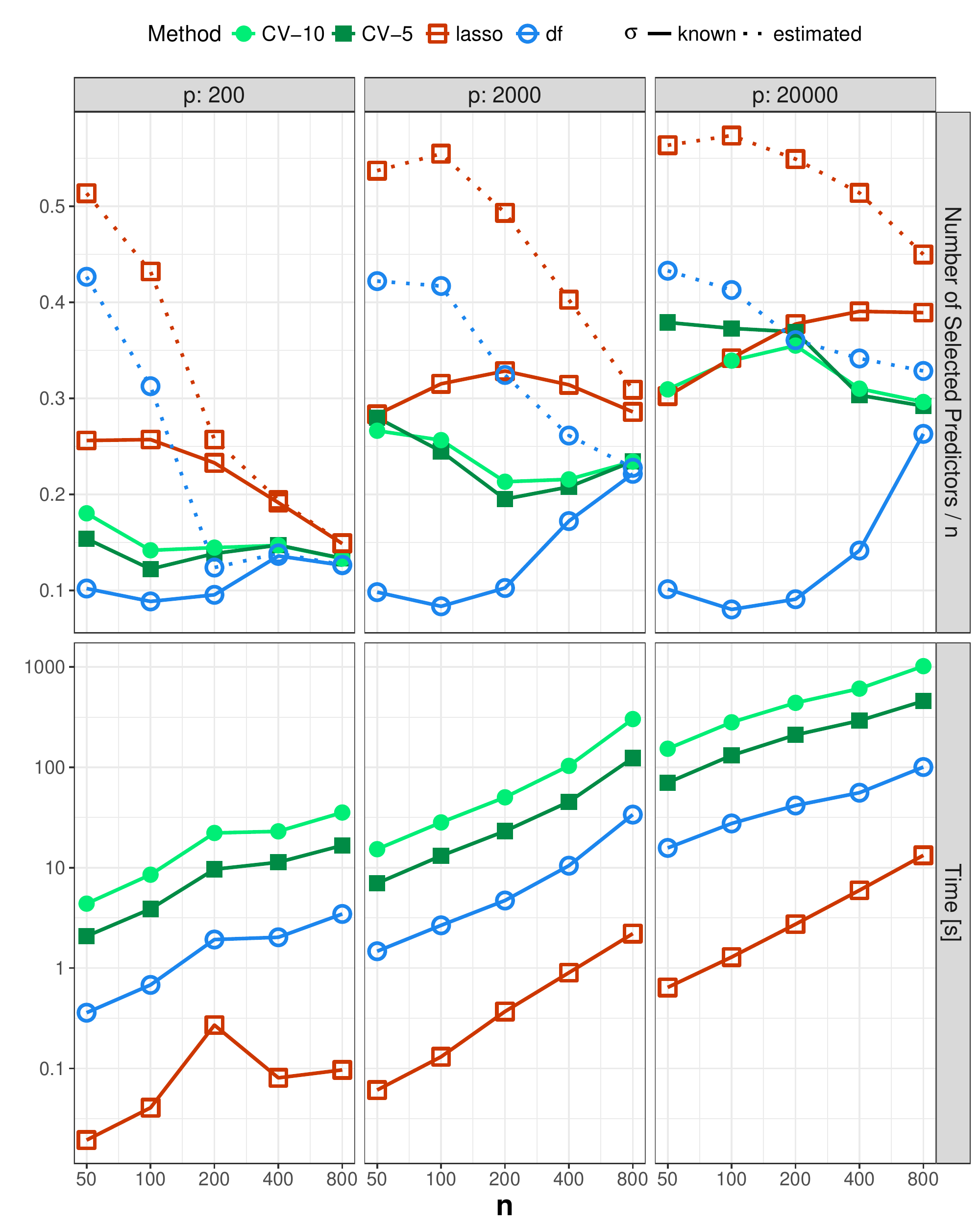}
    	\caption{The number of selected predictors divided by $n$ (top), along with computational time of evaluating the estimator and tuning the $\lambda$-parameter using the different methods (bottom). The design parameters were: $\sigma=0.5$, $\gamma =1$, $\alpha =0.1$,  and the design type was (S) with a constant correlation of $\rho = 0.1$ (see Section 4)}
    	\label{fig:Means}
\end{figure}

\clearpage
\section{Risk estimates}

Plots of the risk estimates relative to $n\sigma^2$ for the estimators $\hat{\mu}_{\mathrm{OLS.l}}^{\hat{\lambda}_{\mathrm{df}_{\mathrm{S}}}}$ and $\hat{\mu}_{\mathrm{OLS.l}}^{\hat{\lambda}_{\mathrm{df}}}$. The dashed lines are the relative risks for the oracle-OLS estimator. 

\includepdf[pages=-]{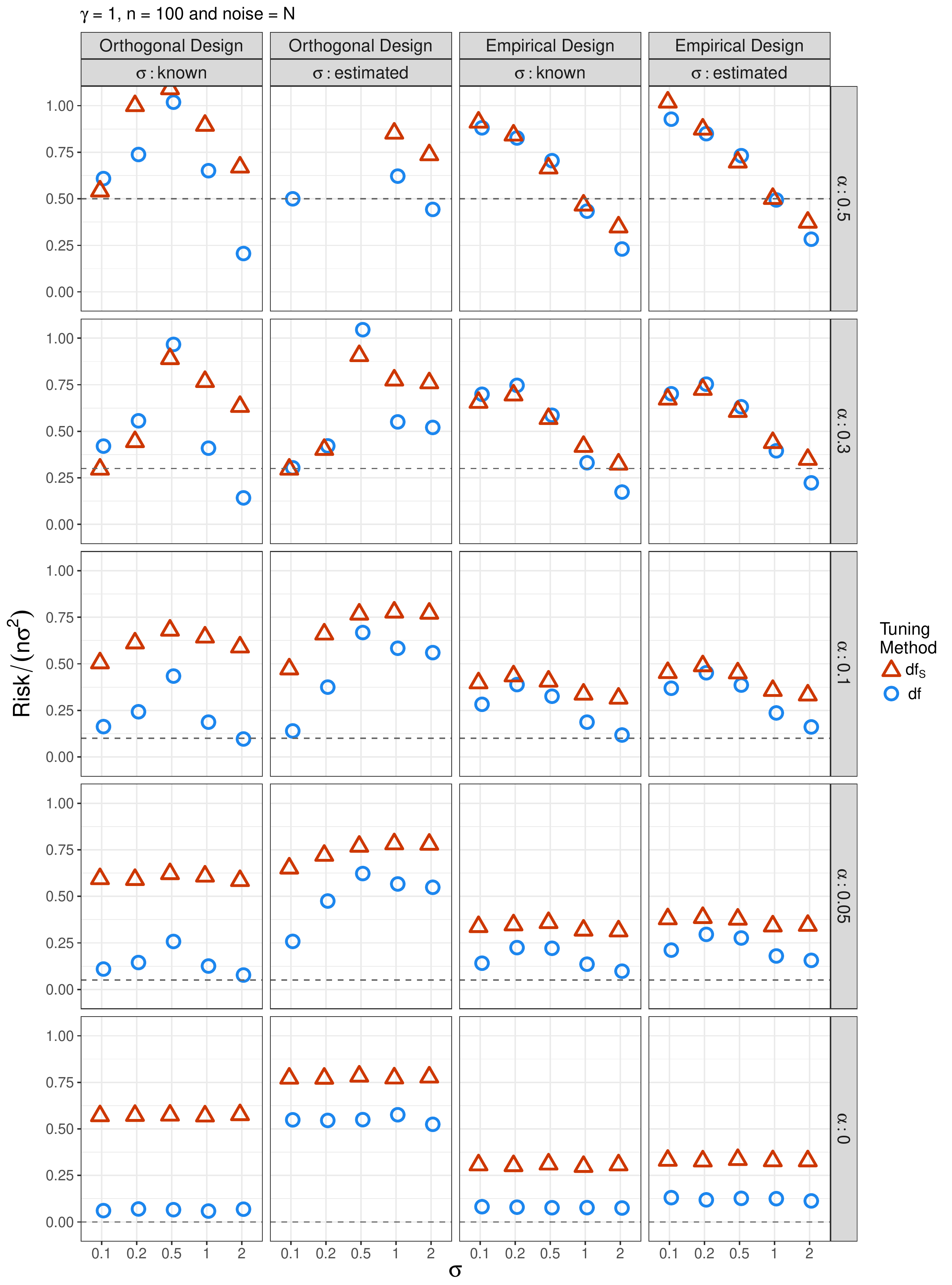}

Plots of the risk estimates relative to $n\sigma^2$ for the estimators $\hat{\mu}_{\mathrm{OLS.l}}^{\hat{\lambda}_{\mathrm{df}}}$, $\hat{\mu}_{\mathrm{OLS.l}}^{\hat{\lambda}_{\mathrm{CV-5}}}$, $\hat{\mu}_{\mathrm{OLS.l}}^{\hat{\lambda}_{\mathrm{CV-10}}}$ and 
$\hat{\mu}_{\mathrm{lasso}}^{\hat{\lambda}_{\mathrm{df}_{\mathrm{S}}}}$. The dashed lines are the relative risks for the oracle-OLS estimator. 

\includepdf[pages=-]{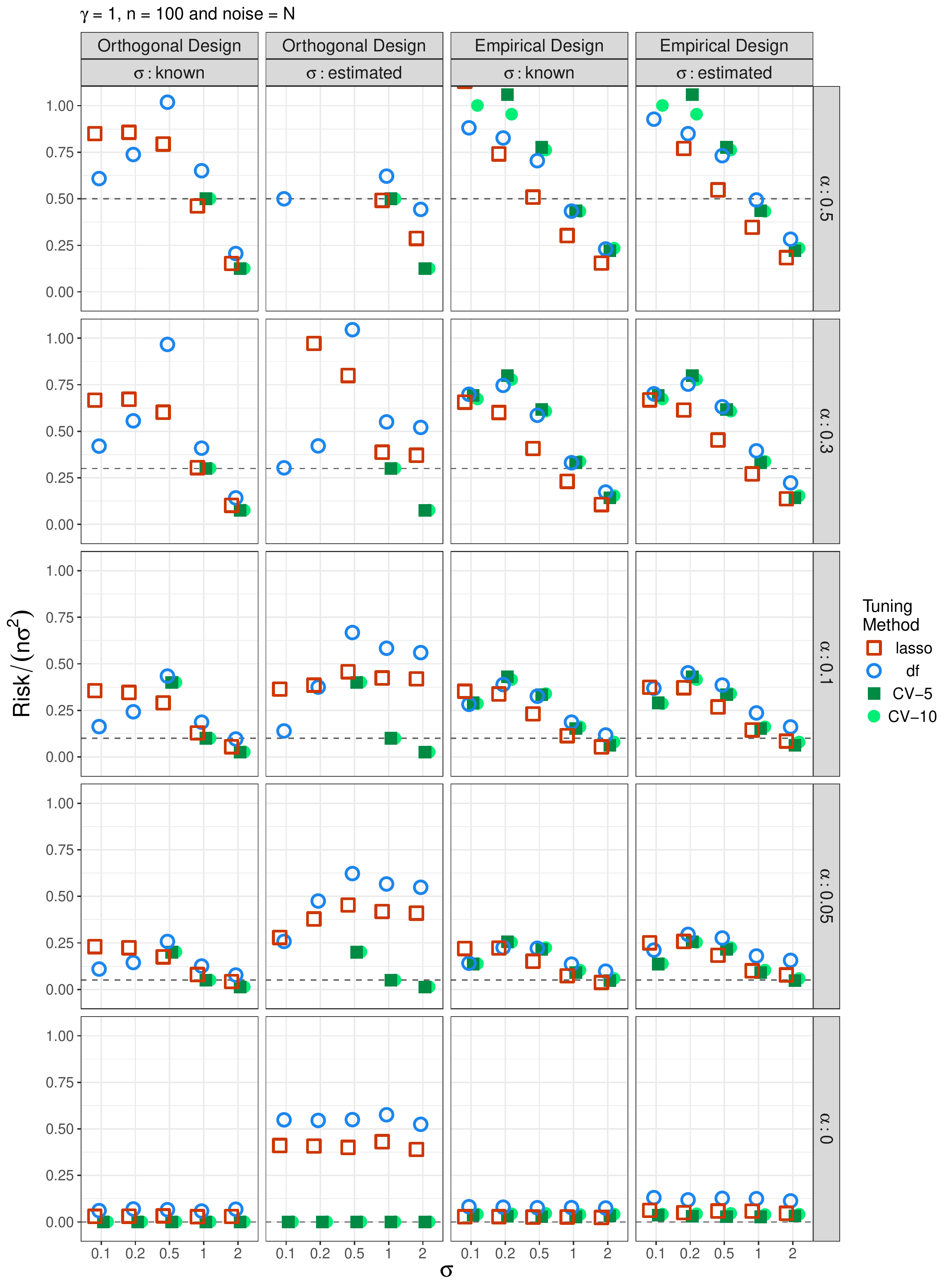}